\documentclass[11pt]{article}
\usepackage[latin1]{inputenc}
\usepackage{amsmath,amsthm,amssymb}
\usepackage{amsfonts}
\usepackage{amsmath,amsthm,amssymb,amscd}
\usepackage{latexsym}
\usepackage{color}
\usepackage{graphicx}
\usepackage{mathrsfs}
\usepackage{cite}
\usepackage[title]{appendix}

\usepackage{color,enumitem,graphicx}
\usepackage[colorlinks=true,urlcolor=black,
citecolor=black,linkcolor=black,linktocpage,pdfpagelabels,
bookmarksnumbered,bookmarksopen]{hyperref}

\makeatletter \@addtoreset{equation}{section} \makeatother

\newtheorem{theorem}{Theorem}[section]

\newtheorem{proposition}{Proposition}[section]
\newtheorem{lemma}{Lemma}[section]
\newtheorem{remark}{Remark}[section]

\allowdisplaybreaks

\begin{document}
\title{Construction of solutions for the critical polyharmonic equation with competing potentials}

\author{Wenjing Chen\footnote{E-mail address:\, {\tt wjchen@swu.edu.cn} (W. Chen), {\tt zxwangmath@163.com} (Z. Wang).}\  \ and Zexi Wang\footnote{Corresponding author.}\\
\footnotesize  School of Mathematics and Statistics, Southwest University,
Chongqing, 400715, P.R. China}

\date{ }
\maketitle

\begin{abstract}
{ In this paper, we consider the following critical polyharmonic equation %in dimension four %$\mathbb{R}^4$
\begin{align*}%\label{abs}
(  -\Delta)^m u+V(|y'|,y'')u=Q(|y'|,y'')u^{m^*-1},\quad u>0, \quad y=(y',y'')\in \mathbb{R}^3\times \mathbb{R}^{N-3},
  \end{align*}
where  $N>4m+1$, $m\in \mathbb{N}^+$, $m^*=\frac{2N}{N-2m}$, $V(|y'|,y'')$ and $Q(|y'|,y'')$ are bounded nonnegative functions in $\mathbb{R}^+\times \mathbb{R}^{N-3}$. By using the reduction argument and local Poho\u{z}aev identities, we prove that if $Q(r,y'')$ has a stable critical point $(r_0,y_0'')$ with $r_0>0$, $Q(r_0,y_0'')>0$, $D^\alpha Q(r_0,y_0'')=0$ for any $|\alpha|\leq 2m-1$ and $B_1V(r_0,y_0'')-B_2\sum\limits_{|\alpha|=2m}D^\alpha Q(r_0,y_0'')\int_{\mathbb{R}^N}y^\alpha U_{0,1}^{m^*}dy>0$,
then the above problem has a faimily of solutions concentrated at points lying on the top and the bottom circles of a cylinder, where $B_1$ and $B_2$ are positive constants that will be given later. }

\vspace{.2cm}

\vspace{.2cm}
\end{abstract}

\section{Introduction}
In this paper, we concern the following polyharmonic  equation with double potentials
\begin{align}\label{pro}
(  -\Delta)^m u+V(y)u=Q(y)u^{m^*-1},\quad u>0,\quad u\in H^{m}(\mathbb{R}^N),
  \end{align}
where $N>4m+1$, $m\in \mathbb{N}^+$, $m^*=\frac{2N}{N-2m}$,  $V(y)$ and $Q(y)$ are two potential functions, and $H^{m}(\mathbb{R}^N)$ is the Hilbert space
%is the completion of $C_0^\infty(\mathbb{R}^N)$ with respect to the norm induced by
with the scalar product
\begin{align*}
 \langle u,v\rangle=\left\{
  \begin{array}{ll}
  \displaystyle \int_{\mathbb{R}^N}\Big((\Delta^{\frac{m}{2}}u)(\Delta^{\frac{m}{2}}v)+V(y)uv \Big)dy,\quad &\text{if $m$ is even},
  \vspace{.3cm}\\ \displaystyle\int_{\mathbb{R}^N}\Big(\big(\nabla (\Delta^{\frac{m-1}{2}}u)\big)\cdot\big(\nabla(\Delta^{\frac{m-1}{2}}v)\big)+V(y)uv\Big) dy,\quad &\text{if $m$ is odd}.
    \end{array}
    \right.
  \end{align*}

In the case of $m=1$, \eqref{pro} reduces to
\begin{align}\label{mpro}
 -\Delta u+V(y)u=Q(y)u^{2^*-1},\quad u>0,\quad u\in H^1(\mathbb{R}^N),
  \end{align}
where $2^*=\frac{2N}{N-2}$.  If $V(y)=0$, \eqref{mpro} can be written as
\begin{align}\label{mmpro}
 -\Delta u=Q(y)u^{2^*-1},\quad u>0,\quad u\in D^{1,2}(\mathbb{R}^N),
  \end{align}
which comes from the following prescribed curvature problem by using the stereographic projection
\begin{equation*}
  -\Delta_{\mathbb{S}^N}u+\frac{N(N-2)}{2}u=\widetilde{Q}(y)u^{2^*-1},\quad u>0\ \ \text{on $\mathbb{S}^N$},
\end{equation*}
where $\widetilde{Q}(y)$ is a potential function.
Recent years, much attention has been paid to study the existence of solutions for \eqref{mmpro} under suitable assumptions on $Q(y)$.
 Using the number of the bubbles of solutions as the parameter, Wei and Yan \cite{WY1} obtained infinitely many solutions on a circle for \eqref{mmpro}  when $Q(y)$ is radially symmetric.
Since then, the idea of \cite{WY1} has been exploited and
further developed by many scholars in the search for infinitely many solutions for equations with critical or asymptotic critical growth, see e.g. \cite{GN,LR,LWX,WY2,MWY,PWW,WW,DMW,DM,GL2,DG,WY3}. In particular, Peng, Wang and Wei \cite{PWW} constructed infinitely many solutions for \eqref{mmpro} under a weaker symmetry condition on $Q(y)$, where the concentration point can be a saddle point of $Q(y)=Q(|\tilde{y}'|,\tilde{y}'')$, $y=(\tilde{y}',\tilde{y}'')\in \mathbb{R}^2\times \mathbb{R}^{N-2}$.
Recently,  Duan, Musso and Wei \cite{DMW} constructed a new type of solutions for \eqref{mmpro}
when $Q(y)$ is radially symmetric, and the solutions concentrate at points lying on the top and the
bottom circles of a cylinder. More precisely, these solutions are different from \cite{PWW,WY1} and have the form
\begin{equation*}
  \sum\limits_{j=1}^kW_{\bar{x}_j,\lambda}+\sum\limits_{j=1}^kW_{\underline{x}_j,\lambda}+\varphi_k,
\end{equation*}
where $W_{x,\lambda}(y)=\big(\frac{\lambda}{1+\lambda^2|y-x|^2}\big)^{\frac{N-2}{2}}$, $\varphi_k$ is a remainder term,
\begin{align*}
   \left\{
  \begin{array}{ll}
  \bar{x}_j=\big(\bar{r}\sqrt{1-\bar{h}^2}\cos \frac{2(j-1)\pi}{k},\bar{r}\sqrt{1-\bar{h}^2}\sin \frac{2(j-1)\pi}{k},\bar{r}\bar{h},0\big),
  \quad &j=1,2,\cdots,k,\\
  \underline{x}_j=\big(\bar{r}\sqrt{1-\bar{h}^2}\cos \frac{2(j-1)\pi}{k},\bar{r}\sqrt{1-\bar{h}^2}\sin \frac{2(j-1)\pi}{k},-\bar{r}\bar{h},0\big),
  \quad &j=1,2,\cdots,k,
    \end{array}
    \right.
 \end{align*}
with $\bar{h}$ goes to zero, and $\bar{r}$ is close to some $r_0>0$.

When $Q(y)=1$, Benci and Cerami \cite{BC} first obtained the existence result for \eqref{mpro}, and proved that if $\|V\|_{\frac{N}{2}}$ %_{L^\frac{N}{2}(\mathbb{R}^N)}$
is suitably small, then
\eqref{mpro} has a solution.
After \cite{BC}, there is no other result for \eqref{mpro} until the contribution made by Chen, Wei and Yan
\cite{CWY}. They proved that when $N \geq5$, $V(y)$ is radially symmetric, and $r^2V(r)$ has a local maximum point or a local minimum point, \eqref{mpro} possesses infinitely many solutions. Later, Peng, Wang and Yan \cite{PWY} generalized the result of \cite{CWY} from $V(y)$ is radial to partial radial. Inspired by the results of \cite{DMW} and \cite{PWY},  under a weaker symmetry condition on $V(y)$, Du et al. \cite{DHWW} obtained a new type of  solutions for \eqref{mpro}, which are different from those obtained in \cite{PWY}. %Moreover, the fractional case was considered in \cite{DHW}.
For more related results about \eqref{mpro}, we
refer the readers to \cite{GG,HWW2,GLN2,VW,DHW} and references therein.

Last but not least, we mention that, by combining the finite dimensional reduction argument and local Poho\u{z}aev identities, He, Wang and Wang \cite{HWW1} studied \eqref{mpro} with double potentials,
 and constructed infinitely many solutions, where they assumed that:

\begin{description}
\item [$(C_1)$] $V(y)=V(|\tilde{y}'|,\tilde{y}'')$ and $Q(y)=Q(|\tilde{y}'|,\tilde{y}'')$ are bounded nonnegative functions, where $y=(\tilde{y}',\tilde{y}'')\in \mathbb{R}^2\times \mathbb{R}^{N-2}$;
\end{description}

\begin{description}
\item [$(C_2)$] $Q(\tilde{r},\tilde{y}'')$ has a stable critical point
$(\tilde{r}_0,\tilde{y}_0'')$ in the sense that
$Q(\tilde{r},\tilde{y}'')$ has a critical point $(\tilde{r}_0,\tilde{y}_0'')$
 satisfying $\tilde{r}_0>0$, $Q(\tilde{r}_0,\tilde{y}_0'')=1$, and
\begin{equation*}
deg\big(\nabla Q(\tilde{r},\tilde{y}''),(\tilde{r}_0,\tilde{y}_0'')\big)\neq0;
\end{equation*}
\end{description}

%\noindent $(A_2)$

\noindent $(C_3)$ $V(\tilde{r},\tilde{y}'')\in C^1(B_\rho(\tilde{r}_0,\tilde{y}_0''))$, $Q(\tilde{r},\tilde{y}'')\in C^{3}(B_\rho(\tilde{r}_0,\tilde{y}_0''))$, and
 \begin{equation*}
   V(\tilde{r}_0,\tilde{y}_0'')\int_{\mathbb{R}^N}U_{0,1}^2dy-\frac{\Delta Q(\tilde{r}_0,\tilde{y}_0'')}{2^*N}\int_{\mathbb{R}^N}y^2U_{0,1}^{2^*}dy>0,
 \end{equation*}
\quad \quad \ where $\rho>0$ is a small constant, and $U_{0,1}$ is the unique positive solution of $- \Delta u=u^{2^*-1}$ in $\mathbb{R}^N$.

%$V(|y'|,y'')$ is a bounded nonnegative function in $\mathbb{R}^3\times \mathbb{R}^{N-3}$
\noindent
More recently, the fractional case was considered  by Liu \cite{LT}.

If $m\geq 2$, \eqref{pro} becomes the equation involving polyharmonic operator. It is worth to pointing that the polyharmonic operator has found considerable interest in the literature due to its geometry roots, and problems involving polyharmonic operator present new and changeling features compared with the elliptic operator (namely $m=1$).
To our best knowledge, few results are known for \eqref{pro} besides \cite{CW,GLN1,GL1,GGH,GL}. In particular,  Guo and Li \cite{GL} first obtained infinitely many solutions for \eqref{pro}, where $V(y)=0$ and $Q(y)$ is radially
symmetric. For $Q(y)=1$ and $V(y)$ is partial radial, we got a new type of solutions in \cite{CW}, which are concentrated at points lying on the top and the bottom circles of a cylinder.
%Then
 %See also \cite{CW,DG} for solutions concentrating at points lying on the top and the bottom circles of a cylinder.
%If $Q(y)=1$,

%Under a weaker symmetry condition on $V(y)$, the existence of infinitely many solutions for problem \eqref{pro} has been established by Guo, Liu and Nie \cite{GLN1}, Guo and Liu \cite{GL1}, respectively.

Motivated by the results mentioned above, especially \cite{CW}, \cite{DMW} and \cite{HWW1}, in this paper, under a weaker symmetry condition on $V(y)$ and $Q(y)$, namely,
 $V(y)\neq0$, $Q(y)\neq1$, and $V(y)$, $Q(y)$ are partial radial,
 we intend to construct a new type of solutions for problem \eqref{pro}, which concentrate at points lying on the top and the bottom circles of a cylinder. Moreover, we would like to
study how the critical points of $V(y)$ and $Q(y)$ affect the existence of solutions. %On the one hand,
%the aim of this paper is to consider \eqref{pro} with a weaker symmetry condition for $V(y)$

Before the statement of the main results, let us first introduce some notations. We assume that $V(y)$ and $Q(y)$ satisfy the
following assumptions:

%Before the statement of the main results, let us first introduce some notations. We assume that $V(y)$ satisfies the following assumptions:

\begin{description}
\item [$(A_1)$] $V(y)=V(|y'|,y'')$ and $Q(y)=Q(|y'|,y'')$ are bounded nonnegative functions, where $y=(y',y'')\in \mathbb{R}^3\times \mathbb{R}^{N-3}$;
\end{description}

\begin{description}
\item [$(A_2)$] $Q(r,y'')$ has a stable critical point
$(r_0,y_0'')$ in the sense that
$Q(r,y'')$ has a critical point $(r_0,y_0'')$
 satisfying $r_0>0$, $Q(r_0,y_0'')=1$, $D^\alpha Q(r_0,y_0'')=0$ for any $|\alpha|\leq 2m-1$, and
\begin{equation*}
deg\big(\nabla Q(r,y''),(r_0,y_0'')\big)\neq0;
\end{equation*}
\end{description}

%\noindent $(A_2)$

\noindent $(A_3)$ $V(r,y'')\in C^1(B_\rho(r_0,y_0''))$, $Q(r,y'')\in C^{2m+1}(B_\rho(r_0,y_0''))$, and
 \begin{equation*}
   B_1V(r_0,y_0'')-B_2\sum\limits_{|\alpha|=2m}D^\alpha Q(r_0,y_0'')\int_{\mathbb{R}^N}y^\alpha U_{0,1}^{m^*}dy>0,
 \end{equation*}
\quad \quad \ where $\rho>0$ is a small constant, $B_1$ and $B_2$ are positive constants given in Lemma \ref{ener1}.

%$V(|y'|,y'')$ is a bounded nonnegative function in $\mathbb{R}^3\times \mathbb{R}^{N-3}$

\vspace{.2cm}

Recall that (see \cite{S})
\begin{equation*}
  U_{x,\lambda}(y)=P_{m,N}^{\frac{N-2m}{4m}}\Big(\frac{\lambda}{1+\lambda^2|y-x|^2}\Big)^{\frac{N-2m}{2}},\quad \lambda>0,\quad x\in \mathbb{R}^N,
\end{equation*}
is the unique radial solution of the equation
\begin{align*}%\label{abs}
(  -\Delta)^m u=u^{m^*-1},\quad u>0,\quad \text{ in $\mathbb{R}^N$},
  \end{align*}
where $P_{m,N}=\prod\limits_{h=-m}^{m-1}(N+2h)$ is a constant.

Define
\begin{align*}
  H_s=\Big\{u:&u\in H^{m}(\mathbb{R}^N),u(y_1,y_2,y_3,y'')=u(y_1,-y_2,-y_3,y''),\\
  &u(r \cos\theta,r \sin\theta, y_3,y'')=u\Big(r\cos \big(\theta+\frac{2j \pi}{k}\big),r\sin \big(\theta+\frac{2j \pi}{k}\big),y_3,y''\Big)\Big\},
\end{align*}
where $r=\sqrt{y_1^2+y_2^2}$ and $\theta=\arctan \frac{y_2}{y_1}$.

 Let
 \begin{align*}
   \left\{
  \begin{array}{ll}
  x_j^+=\big(\bar{r}\sqrt{1-\bar{h}^2}\cos \frac{2(j-1)\pi}{k},\bar{r}\sqrt{1-\bar{h}^2}\sin \frac{2(j-1)\pi}{k},\bar{r}\bar{h},\bar{y}''\big),
  \quad &j=1,2,\cdots,k,\\
  x_j^-=\big(\bar{r}\sqrt{1-\bar{h}^2}\cos \frac{2(j-1)\pi}{k},\bar{r}\sqrt{1-\bar{h}^2}\sin \frac{2(j-1)\pi}{k},-\bar{r}\bar{h},\bar{y}''\big),
  \quad &j=1,2,\cdots,k,
    \end{array}
    \right.
 \end{align*}
where $\bar{y}''$ is a vector in $\mathbb{R}^{N-3}$, $\bar{h}\in (0,1)$ and $(\bar{r},\bar{y}'')$ is close to $(r_0,y_0'')$.

In this paper, we consider the following three cases of $\bar{h}$ in the process of constructing solutions:
\vspace{.2cm}

$\bullet$ {\bf Case 1.} $\bar{h}$ goes to 1;

\vspace{.2cm}

$\bullet$ {\bf Case 2.} $\bar{h}$ is separated from 0 and 1;

\vspace{.2cm}

$\bullet$ {\bf Case 3.} $\bar{h}$ goes to 0.

\vspace{.2cm}
We will use $U_{x_j^{\pm},\lambda}$ as an approximate solution. To deal with the slow decay of this function when $N$ is not big enough, we introduce the smooth cut-off function $\xi(y)=\xi(|y'|,y'')$ satisfying $\xi=1$ if $|(r,y'')-(r_0,y_0'')|\leq \delta$, $\xi=0$ if $|(r,y'')-(r_0,y_0'')|\geq 2\delta$, and $0\leq \xi \leq 1$, where $\delta>0$ is a small constant such that $Q(r,y'')>0$ if $|(r,y'')-(r_0,y_0'')|\leq 10\delta$.

Denote
\begin{equation*}
  Z_{x_j^{\pm},\lambda}=\xi U_{x_j^{\pm},\lambda},\quad Z^*_{\bar{r},\bar{h},\bar{y}'',\lambda}=\sum\limits_{j=1}^kU_{x_j^{+},\lambda}+\sum\limits_{j=1}^kU_{x_j^{-},\lambda},\quad
  Z_{\bar{r},\bar{h},\bar{y}'',\lambda}=\sum\limits_{j=1}^k\xi U_{x_j^{+},\lambda}+\sum\limits_{j=1}^k\xi U_{x_j^{-},\lambda}.
\end{equation*}

As for the {\bf Case 1}, we assume that $\nu=N-4m-\iota$, $\iota>0$ is a small constant, $k>0$ is a large integer, $\lambda\in \big[L_0k^{\frac{N-2m}{N-4m-\nu}},L_1k^{\frac{N-2m}{N-4m-\nu}}\big]$ for some constants $L_1>L_0>0$ and $(\bar{r},\bar{h},\bar{y}'')$ satisfies
\begin{equation}\label{case1}
  |(\bar{r},\bar{y}'')-(r_0,y_0'')|\leq \frac{1}{\lambda^{1-\vartheta}},\quad \sqrt{1-\bar{h}^2}=M_1\lambda^{-\frac{\nu}{N-2m}}+o(\lambda^{-\frac{\nu}{N-2m}}),
\end{equation}
where $\vartheta>0$ is a small constant, $M_1$ is a positive constant. %and $\iota_0$ satisfies
%\begin{equation}\label{satis}
%  \frac{\iota_0}{2m-\iota_0}+\frac{\iota_0}{N-2m}=\frac{N-4m}{N-2m}.
%\end{equation}

\begin{theorem}\label{th1}
Assume that $N>4m+1$, if $V(y)$ and $Q(y)$ satisfy $(A_1)$, $(A_2)$ and $(A_3)$, then there exists an integer $k_0>0$, such that for any $k>k_0$, problem \eqref{pro} has a solution $u_k$ of the form
\begin{equation*}
  u_k=Z_{\bar{r}_k,\bar{h}_k,\bar{y}_k'',\lambda_k}+\phi_k,%=\sum\limits_{j=1}^k\xi U_{x_j^{+},\lambda_k}+\sum\limits_{j=1}^k\xi U_{x_j^{-},\lambda_k}+\phi_k,
\end{equation*}
where $\lambda_k\in \big[L_0k^{\frac{N-2m}{N-4m-\nu}},L_1k^{\frac{N-2m}{N-4m-\nu}}\big]$ and $\phi_k\in H_s$. Moreover, as $k\rightarrow\infty$, $|(\bar{r}_k,\bar{y}_k'')-(r_0,y_0'')|\rightarrow0$, $\sqrt{1-\bar{h}_k^2}=M_1\lambda_k^{-\frac{\nu}{N-2m}}+o(\lambda_k^{-\frac{\nu}{N-2m}})$, and
$\lambda_k^{-\frac{N-2m}{2}}\|\phi_k\|_\infty\rightarrow0$.
\end{theorem}

%\begin{remark}
%{\rm The condition $N>4m+1$ in Theorem \ref{th1} is used in Lemmas \ref{err} and \ref{ener1} to guarantee the existence of a small constant $\iota>0$. }
%\end{remark}

For the {\bf Case 2} and {\bf Case 3}, we assume that $k>0$ is a large integer, $\lambda\in \big[L_0'k^{\frac{N-2m}{N-4m}},L_1'k^{\frac{N-2m}{N-4m}}\big]$ for some constants $L_1'>L_0'>0$ and $(\bar{r},\bar{h},\bar{y}'')$ satisfies
\begin{equation}\label{case2}
  |(\bar{r},\bar{y}'')-(r_0,y_0'')|\leq\frac{1}{\lambda^{1-\vartheta}},\quad \bar{h}=a+{M_2\lambda^{-\frac{N-4m}{N-2m}}}+o(\lambda^{-\frac{N-4m}{N-2m}}),
\end{equation}
where $a\in [0,1)$, $\vartheta>0$ is a small constant, $M_2$ is a positive constant.

\begin{theorem}\label{th2}
Assume that $N>4m+1$, if $V(y)$ and $Q(y)$ satisfy $(A_1)$, $(A_2)$ and $(A_3)$, then there exists an integer $k_0>0$, such that for any $k>k_0$, problem \eqref{pro} has a solution $u_k$ of the form
\begin{equation*}
  u_k=Z_{\bar{r}_k,\bar{h}_k,\bar{y}_k'',\lambda_k}+\phi_k.%=\sum\limits_{j=1}^k\xi U_{x_j^{+},\lambda_k}+\sum\limits_{j=1}^k\xi U_{x_j^{-},\lambda_k}+\phi_k,
\end{equation*}
where $\lambda_k\in \big[L_0k^{\frac{N-2m}{N-4m}},L_1k^{\frac{N-2m}{N-4m}}\big]$ and $\phi_k\in H_s$. Moreover, as $k\rightarrow\infty$, $|(\bar{r}_k,\bar{y}_k'')-(r_0,y_0'')|\rightarrow0$, $\bar{h}_k=a+{M_2\lambda_k^{-\frac{N-4m}{N-2m}}}+o(\lambda_k^{-\frac{N-4m}{N-2m}})$, and
$\lambda_k^{-\frac{N-2m}{2}}\|\phi_k\|_\infty\rightarrow0$.
\end{theorem}

\begin{remark}
{\rm The condition $N>4m+1$ is used in Lemmas \ref{err} and \ref{ener1} to guarantee the existence of a small constant $\iota>0$ for Theorem \ref{th1} ($\iota=0$ in Theorem \ref{th2}).
Moreover, it is needed in the estimates of local Poho\u{z}aev identities \eqref{con1} and \eqref{con2}, we can see this in Lemma \ref{fi}.}
\end{remark}

%\begin{remark}
%{\rm The solutions obtained in Theorems \ref{th1} and \ref{th2} are different from those obtained in \cite{GLN1}.  }
%end{remark}

\begin{remark}
{\rm The condition $D^\alpha Q(r_0,y_0'')=0$ for any $|\alpha|\leq 2m-1$ is used in the estimate of the error term, see Lemma \ref{err} for more details.}
\end{remark}

\begin{remark}
{\rm Combining Theorems \ref{th1} and \ref{th2} with the results obtained in \cite{CW}, we can see that the role of the stable critical point of $Q(r,y'')$ in constructing  solutions is more important than that of $V(r,y'')$. Moreover, $(A_3)$ implies that $V(r_0,y_0'')$ can affect the sign of $B_0:=\sum\limits_{|\alpha|=2m}D^\alpha Q(r_0,y_0'')\int_{\mathbb{R}^N}y^\alpha U_{0,1}^{m^*}dy$ and $B_0$
 can be nonnegative.} %, which is different from that in \cite{GL1},
%where $B_0$ must be negative.}
\end{remark}

%In this paper, we will prove Theorems \ref{th1} and \ref{th2} by using the finite dimensional reduction argument, introduced in \cite{BC1,FW}.
%In the development of this method,
%Wei and Yan \cite{WY1} first found that the number of the bubbles of solutions can be used to as the parameter of construction, and the idea of \cite{WY1} has been exploited and
%further developed in the search for infinitely many solutions for equations with critical or asymptotic critical growth, see e.g. \cite{GN,LR,GL,DM,LWX,WY2,MWY,WY3,CYZ,GL1,PWW,WW,LT}.

The paper is organized as follows. In Section \ref{two}, we will carry out the reduction procedure. Then we will study the reduced problem and prove Theorem \ref{th1} in Section \ref{three}. Section \ref{four} is devoted to the proof of Theorem \ref{th2}. In Appendix \ref{AppA}, we put some basic estimates. Throughout the paper, $C$ denotes positive constant possibly different from line to line, $A=o(B)$ means $A/B\rightarrow 0$ and $A=O(B)$ means that $|A/B|\leq C$.

\section{Reduction argument}\label{two}
Let
\begin{equation*}
  \|u\|_*=\sup\limits_{y\in \mathbb{R}^N}\bigg(\sum\limits_{j=1}^k\Big(\frac{1}{(1+\lambda|y-x_j^+|)^{\frac{N-2m}{2}+\tau}}+
  \frac{1}{(1+\lambda|y-x_j^-|)^{\frac{N-2m}{2}+\tau}}
  \Big)\bigg)^{-1}\lambda^{-\frac{N-2m}{2}}|u(y)|,
\end{equation*}
and
\begin{equation*}
  \|f\|_{**}=\sup\limits_{y\in \mathbb{R}^N}\bigg(\sum\limits_{j=1}^k\Big(\frac{1}{(1+\lambda|y-x_j^+|)^{\frac{N+2m}{2}+\tau}}+
  \frac{1}{(1+\lambda|y-x_j^-|)^{\frac{N+2m}{2}+\tau}}
  \Big)\bigg)^{-1}\lambda^{-\frac{N+2m}{2}}|f(y)|,
\end{equation*}
where $\tau=\frac{N-4m-\nu}{N-2m-\nu}$. %, $\varsigma>0$ is a small constant such that $\tau<1$.
For $j=1,2,\cdots,k$, denote
\begin{equation*}
  Z_{j,2}^{\pm}=\frac{\partial Z_{x_j^\pm,\lambda}}{\partial \lambda},\quad Z_{j,3}^{\pm}=\frac{\partial Z_{x_j^\pm,\lambda}}{\partial \bar{r}},\quad Z_{j,l}^{\pm}=\frac{\partial Z_{x_j^\pm,\lambda}}{\partial \bar{y}_l''},\quad l=4,5,\cdots,N.
\end{equation*}
%Then from Lemma, we have
%\begin{equation*}
%  |Z_{j,2}^{\pm}|\leq C\lambda^{-\beta}Z_{x_j^\pm,\lambda},\quad |Z_{j,l}^{\pm}|\leq C\lambda Z_{x_j^\pm,\lambda},\quad l=3,4,\cdots,N,
%\end{equation*}
%where
%$\beta=\frac{\alpha}{N-2}$ in Theorem, and $\beta=\frac{N-4m}{N-2m}$ in Theorem.

For later calculations, we divide $\mathbb{R}^N$ into $k$ parts, for $j=1,2,\cdots,k$, define
\begin{align*}
  \Omega_j:=\bigg\{y:y=(y_1,y_2,y_3,y'')\in \mathbb{R}^3\times \mathbb{R}^{N-3},
  \Big\langle\frac{(y_1,y_2)}{|(y_1,y_2)|},\Big(\cos \frac{2(j-1)\pi}{k},\sin \frac{2(j-1)\pi}{k}\Big)\Big\rangle_{\mathbb{R}^2}\geq \cos \frac{\pi}{k}\bigg\},
\end{align*}
where $\langle,\rangle_{\mathbb{R}^2}$ denotes the dot product in $\mathbb{R}^2$. For $\Omega_j$, we further divide it into two separate parts
\begin{equation*}
  \Omega_j^+:=\big\{y:y=(y_1,y_2,y_3,y'')\in \Omega_j,y_3\geq0\big\},
\end{equation*}
\begin{equation*}
  \Omega_j^-:=\big\{y:y=(y_1,y_2,y_3,y'')\in \Omega_j,y_3<0\big\}.
\end{equation*}
%It's easy to verify that
%\begin{equation*}
%  \mathbb{R}^N=\bigcup\limits_{j=1}^k\Omega_j,\quad \Omega_j=\Omega_j^+\cup \Omega_j^-,
%\end{equation*}
%and
%\begin{equation*}
%  \Omega_j\cap \Omega_i=\emptyset,\quad \Omega_j^+\cap \Omega_j^-=\emptyset,\quad \text{if $i\neq j$}.
%\end{equation*}

We also define the constrained space
\begin{equation*}
  \mathbb{H}:=\bigg\{v:v\in H_s,\int_{\mathbb{R}^N}Z_{x_j^+,\lambda}^{m^*-2}Z_{j,l}^+vdy=0,
  \int_{\mathbb{R}^N}Z_{x_j^-,\lambda}^{m^*-2}Z_{j,l}^-vdy=0,\ \ j=1,2,\cdots,k,\ \ l=2,3,\cdots,N
  \bigg\}.
\end{equation*}
Consider the following linearized problem %of around $Z_{\bar{r},\bar{h},\bar{y}'',\lambda}$:
\begin{align}\label{lp}
 \left\{
  \begin{array}{ll}
(-\Delta)^m \phi+V(r,y'')\phi-(m^*-1)Q(r,y'')Z_{\bar{r},\bar{h},\bar{y}'',\lambda}^{m^*-2}\phi\\
  \ \ \ =f+\sum\limits_{l=2}^Nc_l\sum\limits_{j=1}^k\Big
  (Z_{x_j^+,\lambda}^{m^*-2}Z_{j,l}^++Z_{x_j^-,\lambda}^{m^*-2}Z_{j,l}^-\Big),
  \quad  \mbox{in $\mathbb{R}^N$},\\
  \phi\in \mathbb{H},
    \end{array}
    \right.
\end{align}
for some real numbers $c_l$. %Then we have

In the sequel of this section, we assume that $(\bar{r},\bar{h},\bar{y}'')$ satisfies \eqref{case1}. Then, we have

\begin{lemma}\label{xian}
Assume that $\phi_k$ solves \eqref{lp} for $f=f_k$. If $\|f_k\|_{**}$ goes to zero as $k$ goes to infinity, so does $\|\phi_k\|_*$.
\end{lemma}
\begin{proof}
Assume by contradiction that there exist $k\rightarrow\infty$, $\lambda_k\in \big[L_0k^{\frac{N-2m}{N-4m-\nu}},L_1k^{\frac{N-2m}{N-4m-\nu}}\big]$, $(\bar{r}_k,\bar{h}_k,\bar{y}_k'')$ satisfying \eqref{case1} and $\phi_k$ solving \eqref{lp} for $f=f_k$, $\lambda=\lambda_k$, $\bar{r}=\bar{r}_k$, $\bar{h}=\bar{h}_k$, $\bar{y}''=\bar{y}_k''$ with $\|f_k\|_{**}\rightarrow 0$ and $\|\phi_k\|_* \geq C>0$. Without loss of generality, we assume that $\|\phi_k\|_*=1$. For simplicity, we drop the subscript $k$.

From \eqref{lp}, we have
\begin{align*}
  |\phi(y)|\leq &C \int_{\mathbb{R}^N}\frac{1}{|y-z|^{N-2m}}Z_{\bar{r},\bar{h},\bar{y}'',\lambda}^{m^*-2}(z)|\phi(z)|dz+C \int_{\mathbb{R}^N}\frac{1}{|y-z|^{N-2m}}|f(z)|dz\\
  &+C \int_{\mathbb{R}^N}\frac{1}{|y-z|^{N-2m}}\Big|\sum\limits_{l=2}^Nc_l\sum\limits_{j=1}^k\Big
  (Z_{x_j^+,\lambda}^{m^*-2}Z_{j,l}^++Z_{x_j^-,\lambda}^{m^*-2}Z_{j,l}^-\Big)\Big|dz.
\end{align*}
The rest of the proof is similar to
\cite[Lemma 2.1]{CW}, so we omit it.
\end{proof}

From Lemma \ref{xian}, applying the same argument as the proof of \cite[Proposition 4.1]{dFM}, we can prove the following result.
\begin{lemma}\label{dc}
There exists an integer $k_0>0$, such that for any $k\geq k_0$ and $f\in L^\infty(\mathbb{R}^N)$, problem \eqref{lp} has a unique solution $\phi=L_k(f)$. Moreover,
\begin{equation*}
  \|L_k(f)\|_*\leq C\|f\|_{**},\quad |c_l|\leq \frac{C}{\lambda^{\eta_l}}\|f\|_{**},
\end{equation*}
where $\eta_2=-\beta_1$, $\beta_1=\frac{\nu}{N-2m}$, $\eta_l=1$ for $l=3,4,\cdots,N$.
\end{lemma}

Now, we consider the following problem
\begin{align}\label{pp}
 \left\{
  \begin{array}{ll}
  \ \ \ (-\Delta)^m (Z_{\bar{r},\bar{h},\bar{y}'',\lambda}+\phi)+V(r,y'')(Z_{\bar{r},\bar{h},\bar{y}'',\lambda}+\phi)\\
  =Q(r,y'')(Z_{\bar{r},\bar{h},\bar{y}'',\lambda}+\phi)_+^{m^*-1}+\sum\limits_{l=2}^Nc_l\sum\limits_{j=1}^k\Big
  (Z_{x_j^+,\lambda}^{m^*-2}Z_{j,l}^++Z_{x_j^-,\lambda}^{m^*-2}Z_{j,l}^-\Big),
  \quad  \mbox{in $\mathbb{R}^N$},\\
  \phi\in \mathbb{H}.
    \end{array}
    \right.
\end{align}

The main result of this section is as follows.
\begin{proposition}\label{fixed}
There exists an integer $k_0>0$, such that for any $k\geq k_0$, $\lambda\in \big[L_0k^{\frac{N-2m}{N-4m-\nu}},L_1k^{\frac{N-2m}{N-4m-\nu}}\big]$, $(\bar{r},\bar{h},\bar{y}'')$ satisfies \eqref{case1}, problem \eqref{pp} has a unique solution $\phi=\phi_{\bar{r},\bar{h},\bar{y}'',\lambda}$ satisfying
\begin{equation*}
  \|\phi\|_*\leq C\big(\frac{1}{\lambda}\big)^{\frac{2m+1-\beta_1}{2}+\varepsilon},\quad |c_l|\leq C\big(\frac{1}{\lambda}\big)^{\frac{2m+1-\beta_1}{2}+\eta_l+\varepsilon},
\end{equation*}
where $\varepsilon>0$ is a small constant.
\end{proposition}

Rewrite \eqref{pp} as
\begin{align}\label{repp}
 \left\{
  \begin{array}{ll}
 \ \ \ (-\Delta)^m \phi+V(r,y'')\phi-(m^*-1)Q(r,y'')Z_{\bar{r},\bar{h},\bar{y}'',\lambda}^{m^*-2}\phi\\
  =N(\phi)+E_k+\sum\limits_{l=2}^Nc_l\sum\limits_{j=1}^k\Big
  (Z_{x_j^+,\lambda}^{m^*-2}Z_{j,l}^++Z_{x_j^-,\lambda}^{m^*-2}Z_{j,l}^-\Big),
  \quad  \mbox{in $\mathbb{R}^N$},\\
  \phi\in \mathbb{H},
    \end{array}
    \right.
\end{align}
where
\begin{equation*}
  N(\phi)=Q(r,y'')\Big((Z_{\bar{r},\bar{h},\bar{y}'',\lambda}+\phi)_+^{m^*-1}-Z_{\bar{r},\bar{h},\bar{y}'',\lambda}^{m^*-1}-(m^*-1)Z_{\bar{r},\bar{h},\bar{y}'',\lambda}^{m^*-2}\phi\Big),
\end{equation*}
and
\begin{align*}
  E_k=&\underbrace{Q(r,y'')Z_{\bar{r},\bar{h},\bar{y}'',\lambda}^{m^*-1}-\sum\limits_{j=1}^k\Big(\xi U_{x_j^+,\lambda}^{m^*-1}+\xi U_{x_j^-,\lambda}^{m^*-1}\Big)}_{:=I_1}-\underbrace{V(r,y'')Z_{\bar{r},\bar{h},\bar{y}'',\lambda}}_{:=I_2}-\underbrace{\sum\limits_{l=0}^{m-1}\binom m l (-\Delta)^{m-l}\xi(-\Delta)^lZ^*_{\bar{r},\bar{h},\bar{y}'',\lambda}}_{:=I_3}\\
  &+\underbrace{2m \sum\limits_{l=0}^{m-1}\binom {m-1} l \nabla \big((-\Delta)^{m-l-1}\xi\big)\cdot \nabla \big((-\Delta)^lZ^*_{\bar{r},\bar{h},\bar{y}'',\lambda}\big)}_{:=I_4}\\
  &+\underbrace{\sum\limits_{l=1}^{m-1}a_l\sum\limits^N_{i_1,i_2,\cdots,i_{m-l}=1}\sum\limits_{s=0}^{l-1}\binom {l-1} s \nabla\Big(\frac{\partial ^{m-l}\big((-\Delta)^sZ^*_{\bar{r},\bar{h},\bar{y}'',\lambda}\big)}{\partial y_{i_{m-l}}\cdots \partial y_{i_1}}\Big)\cdot \nabla \Big(\frac{\partial ^{m-l}\big((-\Delta)^{l-s-1}\xi\big)}{\partial y_{i_{m-l}}\cdots \partial y_{i_1}}\Big)}_{:=I_5},
\end{align*}
where $a_l$, $l=1,2,\cdots,m-1$ are some constants.

In the following, we will make use of the contraction mapping theorem to prove that \eqref{repp} is uniquely solvable under the condition that $\|\phi\|_*$ is small enough, so we need to estimate $N(\phi)$ and $E_k$, respectively.

\begin{lemma}\label{non}
If $N>4m+1$, then
\begin{equation*}
  \|N(\phi)\|_{**}\leq C\|\phi\|_*^{\min\{2,m^*-1\}}.
\end{equation*}
\end{lemma}

\begin{proof}
The proof is similar as that of \cite[Lemma 2.3]{CW}, we omit it here.
  \end{proof}

  \begin{lemma}\label{err}
If $N>4m+1$, then there exists a small constant $\varepsilon>0$ such that
\begin{equation*}
  \|E_k\|_{**}\leq C\big(\frac{1}{\lambda}\big)^{\frac{2m+1-\beta_1}{2}+\varepsilon}.
\end{equation*}
\end{lemma}

\begin{proof}
By symmetry, we assume that $y\in \Omega_1^+$. Then it follows
\begin{equation}\label{da1}
   |y-x_j^-|\geq |y-x_j^+| \geq |y-x_1^+|,\quad j=1,2,\cdots,k,
\end{equation}
\begin{equation}\label{da2}
  |y-x_j^+|\geq C|x_j^+-x_1^+|,\quad |y-x_j^-|\geq C|x_j^--x_1^+|,\quad j=1,2,\cdots,k.
\end{equation}
For $I_1$, we have
\begin{align*}
  I_1=&Q(r,y'')\Big(\sum\limits_{j=1}^k\big(\xi U_{x_j^+,\lambda}+\xi U_{x_j^-,\lambda}\big)\Big)^{m^*-1}-\sum\limits_{j=1}^k\Big(\xi U_{x_j^+,\lambda}^{m^*-1}+\xi U_{x_j^-,\lambda}^{m^*-1}\Big)\\
  =&Q(r,y'')\bigg[\Big(\sum\limits_{j=1}^k\big(\xi U_{x_j^+,\lambda}+\xi U_{x_j^-,\lambda}\big)\Big)^{m^*-1}-\sum\limits_{j=1}^k\Big(\xi U_{x_j^+,\lambda}^{m^*-1}+\xi U_{x_j^-,\lambda}^{m^*-1}\Big)\bigg]\\
  &+\big(Q(r,y'')-1\big)\sum\limits_{j=1}^k\Big(\xi U_{x_j^+,\lambda}^{m^*-1}+\xi U_{x_j^-,\lambda}^{m^*-1}\Big)\\
  :=&I_{11}+I_{12}.
\end{align*}
From \cite[Lemma 2.4]{CW}, we obtain
\begin{equation}\label{err1}
  \|I_{11}\|_{**},\ \ \|I_{2}\|_{**},\ \  \|I_{3}\|_{**},\ \ \|I_{4}\|_{**},\ \ \|I_{5}\|_{**}\leq C\big(\frac{1}{\lambda}\big)^{\frac{2m+1-\beta_1}{2}+\varepsilon}.
  \end{equation}

In the following, we estimate $\|I_{12}\|_{**}$. In the region $|(r,y'')-(r_0,y_0'')|\leq (\frac{1}{\lambda})^{\frac{2m+1-\beta_1}{4m}+\varepsilon}$, using the Taylor's expansion, %by the H\"{o}lder inequality
%\begin{equation*}
%  \sum\limits_{j=1}^ka_j b_j\leq \Big(\sum\limits_{j=1}^ka_j^p\Big)^{\frac{1}{p}} \Big(\sum\limits_{j=1}^kb_j^q\Big)^{\frac{1}{q}},\quad \frac{1}{p}+\frac{1}{q}=1, \quad a_j,b_j\geq0,
%\end{equation*}
we have
\begin{align*}
  |I_{12}|=&\bigg|\Big(\sum\limits_{|\alpha|=2m}\frac{1}{(2m)!}D^\alpha Q(r_0,y_0'')(y-y_0)^\alpha+o\big(|y-y_0|^{2m}\big)\Big)\sum\limits_{j=1}^k\Big(\xi U_{x_j^+,\lambda}^{m^*-1}+\xi U_{x_j^-,\lambda}^{m^*-1}\Big)\bigg|\\
  \leq & C\big(\frac{1}{\lambda}\big)^{\frac{2m+1-\beta_1}{2}+\varepsilon}\lambda^{\frac{N+2m}{2}}\sum\limits_{j=1}^k\Big(\frac{1}{(1+\lambda|y-x_j^+|)^{N+2m}}+
  \frac{1}{(1+\lambda|y-x_j^-|)^{N+2m}}\Big)\\
  \leq & C\big(\frac{1}{\lambda}\big)^{\frac{2m+1-\beta_1}{2}+\varepsilon}\lambda^{\frac{N+2m}{2}}\sum\limits_{j=1}^k\Big(\frac{1}{(1+\lambda|y-x_j^+|)^{\frac{N+2m}{2}+\tau}}+
  \frac{1}{(1+\lambda|y-x_j^-|)^{\frac{N+2m}{2}+\tau}}\Big).
\end{align*}
On the other hand, in the region $(\frac{1}{\lambda})^{\frac{2m+1-\beta_1}{4m}+\varepsilon} \leq |(r,y'')-(r_0,y_0'')|\leq 2\delta$,
\begin{equation*}
  |y-x_j^\pm|\geq |(r,y'')-(r_0,y_0'')|-|(r_0,y_0'')-(\bar{r},\bar{y}'')|\geq (\frac{1}{\lambda})^{\frac{2m+1-\beta_1}{4m}+\varepsilon}-\frac{1}{\lambda^{1-\vartheta}}\geq\frac{1}{2}(\frac{1}{\lambda})^{\frac{2m+1-\beta_1}{4m}+\varepsilon},
\end{equation*}
which leads to
\begin{equation*}
  \frac{1}{1+\lambda |y-x_j^\pm|}\leq C(\frac{1}{\lambda})^{\frac{2m-1+\beta_1}{4m}-\varepsilon},
\end{equation*}
then
\begin{align*}
  |I_{12}|\leq & C\big(\frac{1}{\lambda}\big)^{\frac{2m+1-\beta_1}{2}+\varepsilon}\lambda^{\frac{N+2m}{2}}\sum\limits_{j=1}^k\Big(\frac{1}{(1+\lambda|y-x_j^+|)^{\frac{N+2m}{2}+\tau}}\frac{\lambda^{^{\frac{2m+1-\beta_1}{2}+\varepsilon}}}{(1+\lambda|y-x_j^+|)^{\frac{N+2m}{2}-\tau}}\\
  &\qquad\qquad\qquad\qquad\qquad\ \  +
  \frac{1}{(1+\lambda|y-x_j^-|)^{\frac{N+2m}{2}+\tau}}\frac{\lambda^{{\frac{2m+1-\beta_1}{2}+\varepsilon}}}{(1+\lambda|y-x_j^-|)^{\frac{N+2m}{2}-\tau}}\Big)\\
  \leq  & C\big(\frac{1}{\lambda}\big)^{\frac{2m+1-\beta_1}{2}+\varepsilon}\lambda^{\frac{N+2m}{2}}\lambda^{{\frac{2m+1-\beta_1}{2}+\varepsilon}}(\frac{1}{\lambda})^{(\frac{N+2m}{2}-\tau)(\frac{2m-1+\beta_1}{4m}-\varepsilon)}
  \\
  &\times\sum\limits_{j=1}^k\Big(\frac{1}{(1+\lambda|y-x_j^+|)^{\frac{N+2m}{2}+\tau}}
  +
  \frac{1}{(1+\lambda|y-x_j^-|)^{\frac{N+2m}{2}+\tau}}\Big)\\
  \leq & C\big(\frac{1}{\lambda}\big)^{\frac{2m+1-\beta_1}{2}+\varepsilon}\lambda^{\frac{N+2m}{2}}\sum\limits_{j=1}^k\Big(\frac{1}{(1+\lambda|y-x_j^+|)^{\frac{N+2m}{2}+\tau}}+
  \frac{1}{(1+\lambda|y-x_j^-|)^{\frac{N+2m}{2}+\tau}}\Big),
\end{align*}
where we used the fact that $(\frac{N+2m}{2}-\tau)(\frac{2m-1+\beta_1}{4m}-\varepsilon)\geq {\frac{2m+1-\beta_1}{2}+\varepsilon}$ if $\varepsilon>0$ small enough since $N>4m+1$ and $\iota$ is small. Therefore, we obtain
 \begin{equation}\label{err2}
  \|I_{12}\|_{**}\leq C\big(\frac{1}{\lambda}\big)^{\frac{2m+1-\beta_1}{2}+\varepsilon}.
  \end{equation}
Combining \eqref{err1} and \eqref{err2}, we obtain the result.
\end{proof}

Now we come to prove Proposition \ref{fixed}.

\vspace{.3cm}
\noindent{\bf Proof of Proposition \ref{fixed}.} Denote
\begin{equation*}
  \mathbb{E}=\Big\{\phi:\phi\in C(\mathbb{R}^N)\cap \mathbb{H},\quad \|\phi\|_*\leq C\big(\frac{1}{\lambda}\big)^{\frac{2m+1-\beta_1}{2}}\Big\}.
\end{equation*}
By Lemma \ref{dc}, \eqref{repp} is equivalent to find a fixed point for
\begin{equation}\label{fixp}
  \phi=T(\phi):=L_k(N(\phi)+E_k).
\end{equation}
Hence, it is sufficient to prove that $T$ is a contraction map from $\mathbb{E}$ to $\mathbb{E}$. In fact, for any $\phi\in \mathbb{E}$, by Lemmas \ref{dc}, \ref{non} and \ref{err}, we have
\begin{align*}
  \|T(\phi)\|_*\leq& C\|L_k(N(\phi))\|_{*}+\|L_k(E_k)\|_{*}\leq C\|N(\phi)\|_{**}+C\|E_k\|_{**}\\
  \leq& C\|\phi\|_*^{\min\{2,m^*-1\}}+C\big(\frac{1}{\lambda}\big)^{\frac{2m+1-\beta_1}{2}+\varepsilon}\leq C\big(\frac{1}{\lambda}\big)^{\frac{2m+1-\beta_1}{2}}.
\end{align*}
This shows that $T$ maps from $\mathbb{E}$ to $\mathbb{E}$.

On the other hand, for any $\phi_1,\phi_2\in \mathbb{E}$, we have
\begin{align*}
  \|T(\phi_1)-T(\phi_2)\|_*\leq C\|L_k(N(\phi_1))-L_k(N(\phi_2))\|_{*}\leq C\|N(\phi_1)-N(\phi_2)\|_{**}.
\end{align*}
By symmetry, we assume that $y\in \Omega_1^+$.
If $N\geq 6m$, by \eqref{da2}, the H\"{o}lder inequality and Lemmas \ref{AppA1}, \ref{AppA5},
%\begin{equation*}
%  \sum\limits_{j=1}^ka_j b_j\leq \Big(\sum\limits_{j=1}^ka_j^p\Big)^{\frac{1}{p}} \Big(\sum\limits_{j=1}^kb_j^q\Big)^{\frac{1}{q}},\quad \frac{1}{p}+\frac{1}{q}=1, \quad a_j,b_j\geq0,
%\end{equation*}
 we have
\begin{align*}
  |N(\phi_1)-N(\phi_2)|\leq &C|N'(\phi_1+\kappa(\phi_2-\phi_1))||\phi_1-\phi_2|\leq C\big(|\phi_1|^{m^*-2}+|\phi_2|^{m^*-2}\big)|\phi_1-\phi_2|\\
   \leq& C\big(\|\phi_1\|_*^{m^*-2}+\|\phi_2\|_*^{m^*-2}\big)\|\phi_1-\phi_2\|_*\lambda^{\frac{N+2m}{2}}\\
   &\times\bigg(\sum\limits_{j=1}^k\Big(\frac{1}{(1+\lambda|y-x_j^+|)^{\frac{N-2m}{2}+\tau}}+
  \frac{1}{(1+\lambda|y-x_j^-|)^{\frac{N-2m}{2}+\tau}}\Big)\bigg)^{m^*-1}\\
  \leq &C \big(\|\phi_1\|_*^{m^*-2}+\|\phi_2\|_*^{m^*-2}\big)\|\phi_1-\phi_2\|_*\lambda^{\frac{N+2m}{2}}\\
  &\times\sum\limits_{j=1}^k\Big(\frac{1}{(1+\lambda|y-x_j^+|)^{\frac{N+2m}{2}+\tau}}+
  \frac{1}{(1+\lambda|y-x_j^-|)^{\frac{N+2m}{2}+\tau}}\Big)\\
  &\times \bigg(\sum\limits_{j=1}^k\Big(\frac{1}{(1+\lambda|y-x_j^+|)^{\tau}}+
  \frac{1}{(1+\lambda|y-x_j^-|)^{\tau}}\Big)\bigg)^{m^*-2}\\
  \leq&C \big(\|\phi_1\|_*^{m^*-2}+\|\phi_2\|_*^{m^*-2}\big)\|\phi_1-\phi_2\|_*\lambda^{\frac{N+2m}{2}}\\
  &\times\sum\limits_{j=1}^k\Big(\frac{1}{(1+\lambda|y-x_j^+|)^{\frac{N+2m}{2}+\tau}}+
  \frac{1}{(1+\lambda|y-x_j^-|)^{\frac{N+2m}{2}+\tau}}\Big)\\
  &\times \bigg(1+\sum\limits_{j=2}^k\frac{1}{(\lambda|x_j^+-x_1^+|)^{\tau}}+\sum\limits_{j=1}^k\frac{1}{(\lambda|x_j^--x_1^+|)^{\tau}}\bigg)^{m^*-2}\\
  \leq&C \big(\|\phi_1\|_*^{m^*-2}+\|\phi_2\|_*^{m^*-2}\big)\|\phi_1-\phi_2\|_*\lambda^{\frac{N+2m}{2}}\\
  &\times\sum\limits_{j=1}^k\Big(\frac{1}{(1+\lambda|y-x_j^+|)^{\frac{N+2m}{2}+\tau}}+
  \frac{1}{(1+\lambda|y-x_j^-|)^{\frac{N+2m}{2}+\tau}}\Big),
\end{align*}
that is
\begin{align*}
  \|T(\phi_1)-T(\phi_2)\|_*\leq C \big(\|\phi_1\|_*^{m^*-2}+\|\phi_2\|_*^{m^*-2}\big)\|\phi_1-\phi_2\|_*<\frac{1}{2}\|\phi_1-\phi_2\|_*.
\end{align*}
Therefore, $T$ is a contraction map from $\mathbb{E}$ to $\mathbb{E}$. The case $4m+1<N\leq 6m-1$ can be discussed in a similar way.

By the contraction mapping theorem, there exists a unique $\phi=\phi_{\bar{r},\bar{h},\bar{y}'',\lambda}$ such that \eqref{fixp} holds. Moreover, by Lemmas  \ref{dc}, \ref{non} and \ref{err}, we deduce
\begin{align*}
  \|\phi\|_*\leq C\|L_k(N(\phi))\|_{*}+\|L_k(E_k)\|_{*}\leq C\|N(\phi)\|_{**}+C\|E_k\|_{**}\leq C\big(\frac{1}{\lambda}\big)^{\frac{2m+1-\beta_1}{2}+\varepsilon},
\end{align*}
and
\begin{equation*}
  |c_l|\leq  \frac{C}{\lambda^{\eta_l}}(\|N(\phi)\|_{**}+\|E_k\|_{**})\leq C\big(\frac{1}{\lambda}\big)^{\frac{2m+1-\beta_1}{2}+\eta_l+\varepsilon},
\end{equation*}
for $l=2,3,\cdots,N$. This completes the proof. \qed

\section{Proof of Theorem \ref{th1}}\label{three}
Recall that the functional corresponding to \eqref{pro} is
\begin{align*}
 I(u)=\left\{
  \begin{array}{ll}
  \displaystyle\frac{1}{2}\int_{\mathbb{R}^N}\Big(|\Delta^{\frac{m}{2}}u|^2+V(r,y'')u^2\Big)dy-\frac{1}{m^*}\int_{\mathbb{R}^N}Q(r,y'')(u)_+^{m^*}dy,\quad &\text{if $m$ is even},\vspace{.3cm}\\
  \displaystyle \frac{1}{2}\int_{\mathbb{R}^N}\Big(\big|\nabla \big(\Delta^{\frac{m-1}{2}}u\big)\big|^2+V(r,y'')u^2\Big)dy-\frac{1}{m^*}\int_{\mathbb{R}^N}Q(r,y'')(u)_+^{m^*}dy,\quad &\text{if $m$ is odd}.
    \end{array}
    \right.
  \end{align*}

Let
$\phi=\phi_{\bar{r},\bar{h},\bar{y}'',\lambda}$ be the function obtained in Proposition \ref{fixed} and $u_k=Z_{\bar{r},\bar{h},\bar{y}'',\lambda}+\phi$. In this section, we will choose suitable $(\bar{r},\bar{h},\bar{y}'',\lambda)$ such that $u_k$ is a solution of problem \eqref{pro}. For this purpose, similar to the arguments of \cite[Proposition 3.1]{CW}, we have the following result.
\begin{proposition}
Assume that $(\bar{r},\bar{h},\bar{y}'',\lambda)$ satisfies
\begin{equation}\label{con1}
  \int_{D_\varrho}\big((-\Delta)^mu_k+V(r,y'')u_k-Q(r,y'')(u_k)_+^{m^*-1}\big)\langle y, \nabla u_k\rangle dy=0,
\end{equation}
\begin{equation}\label{con2}
  \int_{D_\varrho}\big((-\Delta)^mu_k+V(r,y'')u_k-Q(r,y'')(u_k)_+^{m^*-1}\big)\frac{\partial u_k}{\partial y_i} dy=0,\quad i=4,5,\cdots,N,
\end{equation}
and
\begin{equation}\label{con3}
  \int_{\mathbb{R}^N}\big((-\Delta)^mu_k+V(r,y'')u_k-Q(r,y'')(u_k)_+^{m^*-1}\big)\frac{\partial Z_{\bar{r},\bar{h},\bar{y}'',\lambda}}{\partial \lambda} dy=0,
\end{equation}
where $u_k=Z_{\bar{r},\bar{h},\bar{y}'',\lambda}+\phi$ and $D_\varrho=\big\{y:y=(y',y''):|(|y'|,y'')-(r_0,y_0'')|\leq \varrho\big\}$ with $\varrho\in (2\delta,5\delta)$, then $c_l=0$ for $l=2,3,\cdots,N$.
\end{proposition}

\begin{lemma}\label{ener1}
We have
\begin{align*}
  &\int_{\mathbb{R}^N}\big((-\Delta)^mu_k+V(r,y'')u_k-Q(r,y'')(u_k)_+^{m^*-1}\big)\frac{\partial Z_{\bar{r},\bar{h},\bar{y}'',\lambda}}{\partial \lambda} dy\\
  =&2k\bigg(-\frac{\tilde{B}_1}{\lambda^{2m+1}}+\sum\limits_{j=2}^k\frac{\tilde{B}_2}{\lambda^{N-2m+1}|x_j^+-x_1^+|^{N-2m}}+
  \sum\limits_{j=1}^k\frac{\tilde{B}_2}{\lambda^{N-2m+1}|x_j^--x_1^+|^{N-2m}}+O\Big(\frac{1}{\lambda^{2m+1+\varepsilon}}\Big)\bigg)\\
  =&2k\bigg(-\frac{\tilde{B}_1}{\lambda^{2m+1}}+\frac{\tilde{B}_3k^{N-2m}}{\lambda^{N-2m+1}(\sqrt{1-\bar{h}^2})^{N-2m}}+
  \frac{\tilde{B}_4 k}{\lambda^{N-2m+1}\bar{h}^{N-2m-1}\sqrt{1-\bar{h}^2}}+O\Big(\frac{1}{\lambda^{2m+1+\varepsilon}}\Big)\bigg),
\end{align*}
where $\tilde{B}_1$, $\tilde{B}_2$, $\tilde{B}_3$, $\tilde{B}_4$ are some positive constants.
\end{lemma}

\begin{proof}
By symmetry, we have
\begin{align*}
  &\int_{\mathbb{R}^N}\big((-\Delta)^mu_k+V(r,y'')u_k-Q(r,y'')(u_k)_+^{m^*-1}\big)\frac{\partial Z_{\bar{r},\bar{h},\bar{y}'',\lambda}}{\partial \lambda} dy\\
  =&\Big\langle I'(Z_{\bar{r},\bar{h},\bar{y}'',\lambda}),\frac{\partial Z_{\bar{r},\bar{h},\bar{y}'',\lambda}}{\partial \lambda}\Big\rangle+2k
  \Big\langle (-\Delta)^m \phi+V(r,y'')\phi-(m^*-1)Q(r,y'')Z_{\bar{r},\bar{h},\bar{y}'',\lambda}^{m^*-2}\phi,\frac{\partial Z_{x_1^+,\lambda}}{\partial \lambda}\Big\rangle\\
  &-\int_{\mathbb{R}^N}Q(r,y'')\Big((Z_{\bar{r},\bar{h},\bar{y}'',\lambda}+\phi)_+^{m^*-1}-Z_{\bar{r},\bar{h},\bar{y}'',\lambda}^{m^*-1}-(m^*-1)Z_{\bar{r},\bar{h},\bar{y}'',\lambda}^{m^*-2}\phi\Big)\frac{\partial Z_{\bar{r},\bar{h},\bar{y}'',\lambda}}{\partial \lambda}  dy\\
  :=&\Big\langle I'(Z_{\bar{r},\bar{h},\bar{y}'',\lambda}),\frac{\partial Z_{\bar{r},\bar{h},\bar{y}'',\lambda}}{\partial \lambda}\Big\rangle+2kI_1-I_2.
\end{align*}
Moreover,
\begin{align*}
  \frac{\partial I(Z_{\bar{r},\bar{h},\bar{y}'',\lambda})}{\partial \lambda}=&\frac{\partial I(Z^*_{\bar{r},\bar{h},\bar{y}'',\lambda})}{\partial \lambda}+O\Big(\frac{k}{\lambda^{2m+1+\varepsilon}}\Big)\\
  =&\int_{\mathbb{R}^N}V(r,y'')Z^*_{\bar{r},\bar{h},\bar{y}'',\lambda} \frac{\partial Z^*_{\bar{r},\bar{h},\bar{y}'',\lambda}}{\partial \lambda}dy +\int_{\mathbb{R}^N}\big(1-Q(r,y'')\big)(Z^*_{\bar{r},\bar{h},\bar{y}'',\lambda})^{m^*-1}
  \frac{\partial Z^*_{\bar{r},\bar{h},\bar{y}'',\lambda}}{\partial \lambda}dy \\& -\int_{\mathbb{R}^N}\Big((Z^*_{\bar{r},\bar{h},\bar{y}'',\lambda})^{m^*-1}-\sum\limits_{j=1}^kU_{x_j^+,\lambda}^{m^*-1}-\sum\limits_{j=1}^kU_{x_j^-,\lambda}^{m^*-1}\Big) \frac{\partial Z^*_{\bar{r},\bar{h},\bar{y}'',\lambda}}{\partial \lambda}dy +O\Big(\frac{k}{\lambda^{2m+1+\varepsilon}}\Big)\\
  :=&I_3+I_4-I_5+O\Big(\frac{k}{\lambda^{2m+1+\varepsilon}}\Big).
\end{align*}
From \cite[Lemmas 3.1, B.1]{CW}, we have
\begin{equation*}
  |I_1|=O\Big(\frac{1}{\lambda^{2m+1+\varepsilon}}\Big),\quad |I_2|=O\Big(\frac{k}{\lambda^{2m+1+\varepsilon}}\Big),\quad I_3=2k\Big(-\frac{{B}_1V(r_0,y_0'')}{\lambda^{2m+1}}+O\Big(\frac{1}{\lambda^{2m+1+\varepsilon}}\Big)\Big),
\end{equation*}
and
\begin{equation*}
  I_5=2k\Big(-\sum\limits_{j=2}^k\frac{{B}_3}{\lambda^{N-2m+1}|x_j^+-x_1^+|^{N-2m}}-
  \sum\limits_{j=1}^k\frac{{B}_3}{\lambda^{N-2m+1}|x_j^--x_1^+|^{N-2m}}+O\Big(\frac{1}{\lambda^{2m+1+\varepsilon}}\Big)\Big),
\end{equation*}
for some positive constants $B_1$ and $B_3$.

In the following, we estimate $I_4$. By symmetry, using Lemmas \ref{AppA1}, \ref{AppA5},  and the Taylor's expansion, we have
\begin{align*}
  I_4=&2k\bigg[\int_{\mathbb{R}^N}\big(1-Q(r,y'')\big)U_{x_1^+,\lambda}^{m^*-1}\frac{\partial U_{x_1^+,\lambda}}{\partial \lambda}dy+O\bigg(\frac{1}{\lambda^{\beta_1}}\int_{\mathbb{R}^N}U_{x_1^+,\lambda}^{m^*-1}\Big(\sum\limits_{j=2}^kU_{x_j^+,\lambda}+\sum\limits_{j=1}^kU_{x_j^-,\lambda}\Big)dy\bigg)
  \bigg]\\
  =&2k\bigg[\int_{B_{\lambda^{-\frac{1}{2}+\varepsilon}}(x_1^+)}\big(1-Q(r,y'')\big)U_{x_1^+,\lambda}^{m^*-1}\frac{\partial U_{x_1^+,\lambda}}{\partial \lambda}dy+
  \int_{B^c_{\lambda^{-\frac{1}{2}+\varepsilon}}(x_1^+)}\big(1-Q(r,y'')\big)U_{x_1^+,\lambda}^{m^*-1}\frac{\partial U_{x_1^+,\lambda}}{\partial \lambda}dy\\
  &+O\bigg(\frac{1}{\lambda^{\beta_1}}\int_{\mathbb{R}^N}U_{x_1^+,\lambda}^{m^*-1}\Big(\sum\limits_{j=2}^kU_{x_j^+,\lambda}+\sum\limits_{j=1}^kU_{x_j^-,\lambda}\Big)dy\bigg)
  \bigg]\\
  =&2k\bigg[\int_{B_{\lambda^{-\frac{1}{2}+\varepsilon}}(x_1^+)}\big(1-Q(r,y'')\big)U_{x_1^+,\lambda}^{m^*-1}\frac{\partial U_{x_1^+,\lambda}}{\partial \lambda}dy+
  O\Big(\frac{1}{\lambda^{\frac{N}{2}+N\varepsilon+\beta_1}}\Big)
  \\&+O\bigg(\frac{1}{\lambda^{\beta_1}}\Big(\sum\limits_{j=2}^k\frac{1}{(\lambda|x_j^+-x_1^+|)^{N-\varepsilon}}+\sum\limits_{j=1}^k\frac{1}{(\lambda|x_j^--x_1^+|)^{N-\varepsilon}}\Big)\bigg)
  \bigg]\\
  =&2k\bigg[\int_{B_{\lambda^{-\frac{1}{2}+\varepsilon}}(x_1^+)}\big(1-Q(r,y'')\big)U_{x_1^+,\lambda}^{m^*-1}\frac{\partial U_{x_1^+,\lambda}}{\partial \lambda}dy
  +O\Big(\frac{1}{\lambda^{\frac{N}{2}+N\varepsilon+\beta_1}}\Big)+O\Big(\frac{1}{\lambda^{\beta_1+\frac{2m}{N-2m}(N-\varepsilon)}}\Big)\bigg]\\
  =&2k\bigg[\int_{B_{\lambda^{-\frac{1}{2}+\varepsilon}}(x_1^+)}\big(1-Q(r,y'')\big)U_{x_1^+,\lambda}^{m^*-1}\frac{\partial U_{x_1^+,\lambda}}{\partial \lambda}dy
  +O\Big(\frac{1}{\lambda^{2m+1+\varepsilon}}\Big)\bigg]\\
  =&2k\bigg[-\sum\limits_{|\alpha|=2m}\frac{1}{(2m)!}D^\alpha Q(r_0,y_0'')\int_{B_{\lambda^{-\frac{1}{2}+\varepsilon}}(x_1^+)}(y-y_0)^\alpha U_{x_1^+,\lambda}^{m^*-1}\frac{\partial U_{x_1^+,\lambda}}{\partial \lambda}dy
  +O\Big(\frac{1}{\lambda^{2m+1+\varepsilon}}\Big)\bigg]\\
  =&2k\bigg[-\sum\limits_{|\alpha|=2m}\frac{1}{(2m)!}D^\alpha Q(r_0,y_0'')\int_{B_{\lambda^{-\frac{1}{2}+\varepsilon}}(0)}(y+x_1^+-y_0)^\alpha U_{0,\lambda}^{m^*-1}\frac{\partial U_{0,\lambda}}{\partial \lambda}dy
  +O\Big(\frac{1}{\lambda^{2m+1+\varepsilon}}\Big)\bigg]\\
  =&2k\bigg[-\sum\limits_{|\alpha|=2m}\frac{1}{(2m)!}D^\alpha Q(r_0,y_0'')\int_{B_{\lambda^{-\frac{1}{2}+\varepsilon}}(0)}(y+x_1^+-y_0)^\alpha \frac{1}{m^*}\frac{\partial U^{m^*}_{0,\lambda}}{\partial \lambda}dy
  +O\Big(\frac{1}{\lambda^{2m+1+\varepsilon}}\Big)\bigg]\\
  =&2k\bigg[-\sum\limits_{|\alpha|=2m}\frac{1}{(2m)!}\frac{1}{m^*}D^\alpha Q(r_0,y_0'')\frac{\partial}{\partial \lambda}\int_{B_{\lambda^{\frac{1}{2}+\varepsilon}}(0)}\big(\frac{y}{\lambda}+x_1^+-y_0\big)^\alpha { U^{m^*}_{0,1}}dy
  +O\Big(\frac{1}{\lambda^{2m+1+\varepsilon}}\Big)\bigg]\\
  =&2k\bigg[\sum\limits_{|\alpha|=2m}\frac{1}{(2m-1)!}\frac{1}{m^*}D^\alpha Q(r_0,y_0'')\frac{1}{ \lambda^{2m+1}}\int_{B_{\lambda^{\frac{1}{2}+\varepsilon}}(0)}y^\alpha { U^{m^*}_{0,1}}dy
  +O\Big(\frac{1}{\lambda^{2m+1+\varepsilon}}\Big)\bigg]\\
  =&2k\bigg[\sum\limits_{|\alpha|=2m}\frac{1}{(2m-1)!}\frac{1}{m^*}D^\alpha Q(r_0,y_0'')\frac{1}{ \lambda^{2m+1}}\int_{\mathbb{R}^N}y^\alpha { U^{m^*}_{0,1}}dy
  +O\Big(\frac{1}{\lambda^{2m+1+\varepsilon}}\Big)\bigg]\\
  =&2k\bigg[\frac{{B}_2}{\lambda^{2m+1}}\sum\limits_{|\alpha|=2m}D^\alpha Q(r_0,y_0'')\int_{\mathbb{R}^N}y^\alpha { U^{m^*}_{0,1}}dy+O\Big(\frac{1}{\lambda^{2m+1+\varepsilon}}\Big)\bigg],
\end{align*}
where $B_2=\frac{1}{(2m-1)!}\frac{1}{m^*}$, and we used the facts that ${\frac{N}{2}+N\varepsilon+\beta_1}\geq 2m+1+\varepsilon$ and $\beta_1+\frac{2m}{N-2m}(N-\varepsilon)\geq 2m+1+\varepsilon$ if $\varepsilon>0$ small enough since $N>4m+1$ and $\iota$ is small.
Using Lemma \ref{AppA6} and the condition $(A_3)$, we derive the conclusion with
\begin{equation*}
  \tilde{B}_1={B}_1V(r_0,y_0'')-B_2\sum\limits_{|\alpha|=2m}D^\alpha Q(r_0,y_0'')\int_{\mathbb{R}^N}y^\alpha { U^{m^*}_{0,1}}dy>0,\quad \tilde{B}_2=B_3.
\end{equation*}
\end{proof}

By Lemma \ref{AppA8}, we have
\begin{equation}\label{con11}
  \int_{D_\varrho}(-\Delta)^mu_k\langle y, \nabla u_k\rangle dy=\frac{1}{2}\int_{\partial D_\varrho}f_m(u_k,u_k) d\sigma-\frac{N-2m}{2}\int_{D_\varrho}u_k(-\Delta)^mu_kdy.
\end{equation}
Integrating by parts, we obtain
\begin{equation}\label{con12}
   \int_{D_\varrho}V(y)u_k\langle y, \nabla u_k\rangle dy= \frac{1}{2}\int_{\partial D_\varrho}\varrho V(y)u_k^2d\sigma-\frac{1}{2} \int_{D_\varrho}\langle y, \nabla V(y)\rangle u_k^2dy- \frac{N}{2}\int_{D_\varrho}V(y)u_k^2 dy,
\end{equation}
and
\begin{align}\label{con13}
   \int_{D_\varrho}Q(y)(u_k)^{m^*-1}_+\langle y, \nabla u_k\rangle dy= &\frac{1}{m^*}\int_{\partial D_\varrho}\varrho Q(y) (u_k)_+^{m^*}d\sigma \nonumber\\
   &-\frac{1}{m^*}\int_{ D_\varrho}\langle y, \nabla Q(y)\rangle(u_k)_+^{m^*}dy- \frac{N}{m^*}\int_{D_\varrho}Q(y)(u_k)_+^{m^*} dy.
\end{align}
Combining \eqref{con11}, \eqref{con12} and \eqref{con13}, we know that \eqref{con1} is equivalent to
\begin{align}\label{trans1}
  &\frac{2m-N}{2}\int_{D_\varrho}u_k(-\Delta)^mu_kdy-\frac{1}{2} \int_{D_\varrho}\langle y, \nabla V(y)\rangle u_k^2dy- \frac{N}{2}\int_{D_\varrho}V(y)u_k^2 dy \nonumber\\
 &+\frac{1}{m^*}\int_{ D_\varrho}\langle y, \nabla Q(y)\rangle(u_k)_+^{m^*}dy+\frac{N}{m^*}\int_{D_\varrho}Q(y)(u_k)_+^{m^*} dy \nonumber\\
  =&O\bigg(\int_{\partial D_\varrho}\Big(\phi^2+|\phi|^{m^*}+\sum\limits_{j=1}^{2m-1}|\nabla^j \phi ||\nabla ^{2m-j}\phi|+\sum\limits_{j=0}^{2m-1}|\nabla^j \phi ||\nabla ^{2m-j-1}\phi|\Big)d\sigma\bigg),
\end{align}
since $u_k=\phi$ on $\partial D_\varrho$.

Similarly, for $i=4,5,\cdots,N$, Lemma \ref{AppA9} shows
\begin{equation}\label{con21}
  \int_{D_\varrho}(-\Delta)^mu_k\frac{\partial u_k}{\partial y_i} dy=\frac{1}{2}\int_{\partial D_\varrho}g_{m,i}(u_k,u_k) d\sigma.
\end{equation}
Integrating by parts, we have
\begin{equation}\label{con22}
   \int_{D_\varrho}V(r,y'')u_k\frac{\partial u_k}{\partial y_i} dy= \frac{1}{2}\int_{\partial D_\varrho} V(r,y'')u_k^2 \nu_i d\sigma-\frac{1}{2} \int_{D_\varrho}\frac{\partial V(r,y'')}{\partial y_i}u_k^2dy,
\end{equation}
and
\begin{equation}\label{con23}
   \int_{D_\varrho}Q(r,y'')(u_k)^{m^*-1}_+\frac{\partial u_k}{\partial y_i} dy= \frac{1}{m^*}\int_{\partial D_\varrho}Q(r,y'') (u_k)_+^{m^*}\nu_id\sigma-\frac{1}{m^*}\int_{ D_\varrho} \frac{\partial Q(r,y'')}{\partial y_i}(u_k)_+^{m^*}dy,
\end{equation}
 where $\nu=(\nu_1,\nu_2,\cdots,\nu_N)$ denotes the outward unit normal vector of $\partial D_\varrho$.
Combining \eqref{con21}, \eqref{con22} and \eqref{con23}, we know that \eqref{con2} is equivalent to
\begin{align}\label{trans2}
  &\int_{D_\varrho}\frac{\partial V(r,y'')}{\partial y_i}u_k^2dy-\frac{2}{m^*}\int_{ D_\varrho} \frac{\partial Q(r,y'')}{\partial y_i}(u_k)_+^{m^*}dy \nonumber\\
  =&O\bigg(\int_{\partial D_\varrho}\Big(\phi^2+|\phi|^{m^*}+\sum\limits_{j=1}^{2m-1}|\nabla^j \phi ||\nabla ^{2m-j}\phi|\Big)d\sigma\bigg),\quad i=4,5,\cdots,N.
\end{align}

Multiplying \eqref{pp} by $u_k$ and integrating in $D_\varrho$, we obtain
\begin{align*}
  &\int_{D_\varrho}u_k(-\Delta)^mu_kdy+\int_{D_\varrho}V(y)u_k^2 dy
 \\ =&\int_{D_\varrho}Q(y)(u_k)_+^{m^*} dy+\sum\limits_{l=2}^Nc_l\sum\limits_{j=1}^k\int_{D_\varrho}\Big
  (Z_{x_j^+,\lambda}^{m^*-2}Z_{j,l}^++Z_{x_j^-,\lambda}^{m^*-2}Z_{j,l}^-\Big)Z_{\bar{r},\bar{h},\bar{y}'',\lambda}dy.
\end{align*}
Thus, \eqref{trans1} can be reduced to
\begin{align}\label{trans1'}
  &m\int_{D_\varrho}V(y)u_k^2 dy+\frac{1}{2} \int_{D_\varrho}\langle y, \nabla V(y)\rangle u_k^2dy-\frac{1}{m^*}\int_{ D_\varrho} \langle y, \nabla Q(y)(u_k)_+^{m^*}dy\nonumber\\
  =&\frac{2m-N}{2}\sum\limits_{l=2}^Nc_l\sum\limits_{j=1}^k\int_{D_\varrho}\Big
  (Z_{x_j^+,\lambda}^{m^*-2}Z_{j,l}^++Z_{x_j^-,\lambda}^{m^*-2}Z_{j,l}^-\Big)Z_{\bar{r},\bar{h},\bar{y}'',\lambda}dy \nonumber\\
  &+O\bigg(\int_{\partial D_\varrho}\Big(\phi^2+|\phi|^{m^*}+\sum\limits_{j=1}^{2m-1}|\nabla^j \phi ||\nabla ^{2m-j}\phi|+\sum\limits_{j=0}^{2m-1}|\nabla^j \phi ||\nabla ^{2m-j-1}\phi|\Big)d\sigma\bigg).
\end{align}
Using \eqref{trans2}, we can rewrite \eqref{trans1'} as
\begin{align}\label{trans1''}
  &m\int_{D_\varrho}V(y)u_k^2 dy+\frac{1}{2} \int_{D_\varrho}r \frac{\partial V(r,y'')}{\partial r} u_k^2dy-\frac{1}{m^*} \int_{D_\varrho}r \frac{\partial Q(r,y'')}{\partial r} (u_k)_+^{m^*}dy\nonumber\\
  =&\frac{2m-N}{2}\sum\limits_{l=2}^Nc_l\sum\limits_{j=1}^k\int_{D_\varrho}\Big
  (Z_{x_j^+,\lambda}^{m^*-2}Z_{j,l}^++Z_{x_j^-,\lambda}^{m^*-2}Z_{j,l}^-\Big)Z_{\bar{r},\bar{h},\bar{y}'',\lambda}dy \nonumber\\
  &+O\bigg(\int_{\partial D_\varrho}\Big(\phi^2+|\phi|^{m^*}+\sum\limits_{j=1}^{2m-1}|\nabla^j \phi ||\nabla ^{2m-j}\phi|+\sum\limits_{j=0}^{2m-1}|\nabla^j \phi ||\nabla ^{2m-j-1}\phi|\Big)d\sigma\bigg).
\end{align}
A direct computation gives
\begin{equation*}
  \sum\limits_{j=1}^k\int_{D_\varrho}\Big
  (Z_{x_j^+,\lambda}^{m^*-2}Z_{j,l}^++Z_{x_j^-,\lambda}^{m^*-2}Z_{j,l}^-\Big)Z_{\bar{r},\bar{h},\bar{y}'',\lambda}dy=O\Big(\frac{k\lambda^{\eta_l}}{\lambda^{2m}}\Big),
\end{equation*}
this with Proposition \ref{fixed} yields
\begin{equation*}
  \sum\limits_{l=2}^Nc_l\sum\limits_{j=1}^k\int_{D_\varrho}\Big
  (Z_{x_j^+,\lambda}^{m^*-2}Z_{j,l}^++Z_{x_j^-,\lambda}^{m^*-2}Z_{j,l}^-\Big)Z_{\bar{r},\bar{h},\bar{y}'',\lambda}dy=O\Big(\frac{k}{\lambda^{3m+\frac{1-\beta_1}{2}+\varepsilon}}\Big)=
  o\Big(\frac{k}{\lambda^{2m}}\Big).
\end{equation*}
Therefore, \eqref{trans1''} is equivalent to
\begin{align}\label{trans1'''}
  & \int_{D_\varrho}\frac{1}{2r^{2m-1}} \frac{\partial\big(r^{2m} V(r,y'')\big)}{\partial r} u_k^2dy-\frac{1}{m^*} \int_{D_\varrho}r \frac{\partial Q(r,y'')}{\partial r} (u_k)_+^{m^*}dy\nonumber\\
  =& o\Big(\frac{k}{\lambda^{2m}}\Big)
  +O\bigg(\int_{\partial D_\varrho}\Big(\phi^2+|\phi|^{m^*}+\sum\limits_{j=1}^{2m-1}|\nabla^j \phi ||\nabla ^{2m-j}\phi|+\sum\limits_{j=0}^{2m-1}|\nabla^j \phi ||\nabla ^{2m-j-1}\phi|\Big)d\sigma\bigg).
\end{align}

First, we estimate \eqref{trans2} and \eqref{trans1'''} from the right hand, and it is sufficient to estimate
\begin{equation*}
  \int_{D_{4\delta}\backslash D_{3\delta}}\Big(\phi^2+|\phi|^{m^*}+\sum\limits_{j=1}^{2m-1}|\nabla^j \phi ||\nabla ^{2m-j}\phi|+\sum\limits_{j=0}^{2m-1}|\nabla^j \phi ||\nabla ^{2m-j-1}\phi|\Big)dy.
\end{equation*}
We first prove
\begin{lemma}\label{fi}
It holds
\begin{equation*}
  \int_{\mathbb{R}^N}\Big(\phi(-\Delta)^m \phi+V(r,y'')\phi^2\Big)dy=O\Big(\frac{k}{\lambda^{2m+1-\beta_1+\varepsilon}}\Big).
\end{equation*}
\end{lemma}
\begin{proof}
Multiplying \eqref{pp} by $\phi$ and integrating in $\mathbb{R}^N$, we have
\begin{align*}
  &\int_{\mathbb{R}^N}\Big(\phi(-\Delta)^m \phi+V(r,y'')\phi^2\Big)dy\\=&\int_{\mathbb{R}^N}\Big(Q(r,y'')(Z_{\bar{r},\bar{h},\bar{y}'',\lambda}+\phi)_+^{m^*-1}-V(r,y'')Z_{\bar{r},\bar{h},\bar{y}'',\lambda}-
  (-\Delta)^mZ_{\bar{r},\bar{h},\bar{y}'',\lambda}\Big)\phi dy\\
  =&\int_{\mathbb{R}^N}Q(r,y'')\Big((Z_{\bar{r},\bar{h},\bar{y}'',\lambda}+\phi)_+^{m^*-1}-Z_{\bar{r},\bar{h},\bar{y}'',\lambda}^{m^*-1}\Big)\phi dy
  +\int_{\mathbb{R}^N}\big(Q(r,y'')-1\big)Z_{\bar{r},\bar{h},\bar{y}'',\lambda}^{m^*-1}\phi dy
  \\&-\int_{\mathbb{R}^N}V(r,y'')Z_{\bar{r},\bar{h},\bar{y}'',\lambda}
  \phi dy+\int_{\mathbb{R}^N}\Big(Z_{\bar{r},\bar{h},\bar{y}'',\lambda}^{m^*-1}-(-\Delta)^mZ_{\bar{r},\bar{h},\bar{y}'',\lambda}\Big)\phi dy\\
  :=&I_1+I_2-I_3+I_4.
\end{align*}
From \cite[Lemma 3.2]{CW}, we obtain
\begin{equation*}
  |I_1|=O\Big(\frac{k}{\lambda^{2m+1-\beta_1+\varepsilon}}\Big),\quad |I_3|=O\Big(\frac{k}{\lambda^{2m+1-\beta_1+\varepsilon}}\Big),\quad |I_4|=O\Big(\frac{k}{\lambda^{2m+1-\beta_1+\varepsilon}}\Big).
\end{equation*}

In the following, we estimate $I_2$.  By symmetry, using Lemmas \ref{AppA1}, \ref{AppA5}, and the
Taylor's expansion, we have
\begin{align*}
  |I_2|\leq &C\|\phi\|_*\lambda^N\int_{\mathbb{R}^N}\big|Q(r,y'')-1\big|\bigg(\sum\limits_{j=1}^k\Big(\frac{1}{(1+\lambda|y-x_j^+|)^{{N-2m}}}+\frac{1}{(1+\lambda|y-x_j^-|)^{{N-2m}}}\Big)\bigg)^{m^*-1}\\
  &\times \sum\limits_{j=1}^k\Big(\frac{1}{(1+\lambda|y-x_j^+|)^{\frac{N-2m}{2}+\tau}}+\frac{1}{(1+\lambda|y-x_j^-|)^{\frac{N-2m}{2}+\tau}}\Big)dy\\
  \leq &C\|\phi\|_*\lambda^N\int_{\mathbb{R}^N}\big|Q(r,y'')-1\big|\sum\limits_{j=1}^k\Big(\frac{1}{(1+\lambda|y-x_j^+|)^{{N+2m}}}+\frac{1}{(1+\lambda|y-x_j^-|)^{{N+2m}}}\Big)\\
  &\times \sum\limits_{j=1}^k\Big(\frac{1}{(1+\lambda|y-x_j^+|)^{\frac{N-2m}{2}+\tau}}+\frac{1}{(1+\lambda|y-x_j^-|)^{\frac{N-2m}{2}+\tau}}\Big)dy\\
  \leq &Ck\|\phi\|_*\lambda^N\left\{\int_{B_{\lambda^{-\frac{1}{2}+\varepsilon}}(x_1^+)}+\int_{B^c_{\lambda^{-\frac{1}{2}+\varepsilon}}(x_1^+)}\right\}\big|Q(r,y'')-1\big|
  \frac{1}{(1+\lambda|y-x_1^+|)^{{N+2m}}}\\
  &\times \sum\limits_{j=1}^k\Big(\frac{1}{(1+\lambda|y-x_j^+|)^{\frac{N-2m}{2}+\tau}}+\frac{1}{(1+\lambda|y-x_j^-|)^{\frac{N-2m}{2}+\tau}}\Big)dy\\
  \leq &Ck\|\phi\|_*\lambda^N\int_{B_{\lambda^{-\frac{1}{2}+\varepsilon}}(x_1^+)}\big|Q(r,y'')-1\big|
  \frac{1}{(1+\lambda|y-x_1^+|)^{{N+2m}}}\\
  &\times \sum\limits_{j=1}^k\Big(\frac{1}{(1+\lambda|y-x_j^+|)^{\frac{N-2m}{2}+\tau}}+\frac{1}{(1+\lambda|y-x_j^-|)^{\frac{N-2m}{2}+\tau}}\Big)dy+Ck\big(\frac{1}{\lambda}\big)^
  {(\frac{1}{2}+\varepsilon)(\frac{N}{2}+m)+\frac{2m+1-\beta_1}{2}+\varepsilon}\\
  \leq &Ck\|\phi\|_*\lambda^N\int_{B_{\lambda^{-\frac{1}{2}+\varepsilon}}(x_1^+)}\big|Q(r,y'')-1\big|
  \frac{1}{(1+\lambda|y-x_1^+|)^{{N+2m}}}\\
  &\times \sum\limits_{j=1}^k\Big(\frac{1}{(1+\lambda|y-x_j^+|)^{\frac{N-2m}{2}+\tau}}+\frac{1}{(1+\lambda|y-x_j^-|)^{\frac{N-2m}{2}+\tau}}\Big)dy+\frac{Ck}{\lambda^{2m+1-\beta_1+\varepsilon}}\\
  \leq &Ck\|\phi\|_*\lambda^N\left|\sum\limits_{|\alpha|=2m}\frac{1}{(2m)!}D^\alpha Q(r_0,y_0'')\right|\int_{B_{\lambda^{-\frac{1}{2}+\varepsilon}}(x_1^+)}|y-y_0|^\alpha
  \frac{1}{(1+\lambda|y-x_1^+|)^{{N+2m}}}\\
  &\times \sum\limits_{j=1}^k\Big(\frac{1}{(1+\lambda|y-x_j^+|)^{\frac{N-2m}{2}+\tau}}+\frac{1}{(1+\lambda|y-x_j^-|)^{\frac{N-2m}{2}+\tau}}\Big)dy+\frac{Ck}{\lambda^{2m+1-\beta_1+\varepsilon}}\\
  \leq &Ck\|\phi\|_*\left|\sum\limits_{|\alpha|=2m}\frac{1}{(2m)!}D^\alpha Q(r_0,y_0'')\right|\int_{B_{\lambda^{\frac{1}{2}+\varepsilon}}(0)}\Big|\frac{y}{\lambda}+x_1^+-y_0\Big|^\alpha
  \frac{1}{(1+|y|)^{{N+2m}}}\\
  &\times \sum\limits_{j=1}^k\Big(\frac{1}{(1+|y+\lambda(x_1^+-x_j^+)|)^{\frac{N-2m}{2}+\tau}}+\frac{1}{(1+|y+\lambda(x_1^+-x_j^-)|)^{\frac{N-2m}{2}+\tau}}\Big)dy+\frac{Ck}{\lambda^{2m+1-\beta_1+\varepsilon}}\\
  \leq &C\frac{k\|\phi\|_*}{\lambda^{2m}}+\frac{Ck}{\lambda^{2m+1-\beta_1+\varepsilon}}=O\Big(\frac{k}{\lambda^{2m+1-\beta_1+\varepsilon}}\Big),
\end{align*}
where we used the fact that $(\frac{1}{2}+\varepsilon)(\frac{N}{2}+m)+\frac{2m+1-\beta_1}{2}+\varepsilon\geq 2m+1-\beta_1+\varepsilon$ if $\varepsilon>0$ small enough since $N>4m+1$.
This completes the proof.
\end{proof}

With a similar argument of \cite{CW}, we have the following lemmas.
\begin{lemma}\label{se}
It holds
\begin{equation*}
  \int_{D_{4\delta}\backslash D_{3\delta}}\big(\phi^2+|\phi|^{m^*}\big)dy=O\Big(\frac{k}{\lambda^{2m+1-\beta_1+\varepsilon}}\Big).
\end{equation*}
\end{lemma}

\begin{lemma}\label{thi}
It holds
\begin{equation*}
  \|\phi\|_{C^{2m-1}(D_{4\delta}\backslash D_{3\delta})}\leq \frac{C}{\lambda^{m+\tau+\varepsilon}}.
\end{equation*}
\end{lemma}

From Lemmas \ref{fi}, \ref{se} and \ref{thi}, we know
\begin{equation*}
  \int_{D_{4\delta}\backslash D_{3\delta}}\Big(\phi^2+|\phi|^{m^*}+\sum\limits_{j=1}^{2m-1}|\nabla^j \phi ||\nabla ^{2m-j}\phi|+\sum\limits_{j=0}^{2m-1}|\nabla^j \phi ||\nabla ^{2m-j-1}\phi|\Big)dy=O\Big(\frac{k}{\lambda^{2m+1-\beta_1+\varepsilon}}\Big)=o\Big(\frac{k}{\lambda^{2m}}\Big).
\end{equation*}
Thus, there exists $\varrho\in (3\delta,4\delta)$ such that
\begin{equation}\label{trans3}
  \int_{\partial D_{\varrho}}\Big(\phi^2+|\phi|^{m^*}+\sum\limits_{j=1}^{2m-1}|\nabla^j \phi ||\nabla ^{2m-j}\phi|+\sum\limits_{j=0}^{2m-1}|\nabla^j \phi ||\nabla ^{2m-j-1}\phi|\Big)d\sigma=o\Big(\frac{k}{\lambda^{2m}}\Big).
\end{equation}

Conversely, we need to estimate \eqref{trans2} and \eqref{trans1'''} from the left hand, and we have the following lemma.
\begin{lemma}\cite[Lemma 3.5]{CW}\label{converse}
For any function $h(r,y'')\in C^1(\mathbb{R}^{N-2},\mathbb{R})$, there holds
\begin{equation*}
  \int_{ D_{\varrho}}h(r,y'')u_k^2dy=2k\Big(\frac{1}{\lambda^{2m}}h(\bar{r},\bar{y}'')\int_{\mathbb{R}^N}U_{0,1}^2dy+o\big(\frac{1}{\lambda^{2m}}\big)\Big).
\end{equation*}
\end{lemma}

\begin{lemma}\label{conversese}
For any function $h(r,y'')\in C^1(\mathbb{R}^{N-2},\mathbb{R})$, there holds
\begin{equation*}
  \int_{ D_{\varrho}}h(r,y'')(u_k)_+^{m^*}dy=2k\Big(h(\bar{r},\bar{y}'')\int_{\mathbb{R}^N}U_{0,1}^{m^*}dy+o\big(\frac{1}{\lambda^{1-\varepsilon}}\big)\Big).
\end{equation*}
\end{lemma}
\begin{proof}
Since $u_k=Z_{\bar{r},\bar{h},\bar{y}'',\lambda}+\phi$, we have
\begin{align*}
  \int_{ D_{\varrho}}h(r,y'')(u_k)_+^{m^*}dy=&\int_{ D_{\varrho}}h(r,y'')Z_{\bar{r},\bar{h},\bar{y}'',\lambda}^{m^*}dy+O\Big(\int_{ D_{\varrho}}|\phi|^{m^*}dy\Big)
 \\& +O\Big(\int_{ D_{\varrho}}Z_{\bar{r},\bar{h},\bar{y}'',\lambda} |\phi|^{m^*-1} dy\Big)+O\Big(\int_{ D_{\varrho}}Z_{\bar{r},\bar{h},\bar{y}'',\lambda}^{m^*-1} |\phi| dy\Big)\\
 :=&I_1+I_2+I_3+I_4.
\end{align*}
For $I_1$, a direct computation leads to
\begin{equation*}
  I_1=2k\Big(h(\bar{r},\bar{y}'')\int_{\mathbb{R}^N}U_{0,1}^{m^*}dy+o\big(\frac{1}{\lambda^{1-\varepsilon}}\big)\Big).
\end{equation*}

For $I_2$, from the proof of Lemma \ref{se}, we have
\begin{equation*}
  I_2=O\Big(\frac{k}{\lambda^{2m+1-\beta_1+\varepsilon}}\Big)=o\big(\frac{k}{\lambda^{1-\varepsilon}}\big).
\end{equation*}

By symmetry, we obtain
\begin{align*}
  I_3
  \leq &C\|\phi\|_*^{m^*-1} \lambda^{N}\int_{ \mathbb{R}^N}\sum\limits_{j=1}^k\Big(\frac{1}{(1+\lambda|y-x_j^+|)^{{N-2m}}}+\frac{1}{(1+\lambda|y-x_j^-|)^{{N-2m}}}\Big)\\
  &\times \bigg(\sum\limits_{j=1}^k\Big(\frac{1}{(1+\lambda|y-x_j^+|)^{\frac{N-2m}{2}+\tau}}+\frac{1}{(1+\lambda|y-x_j^-|)^{\frac{N-2m}{2}+\tau}}\Big)\bigg)^{m^*-1}dy\\
  \leq
  &C\|\phi\|_*^{m^*-1} \lambda^{N}\int_{ \mathbb{R}^N}\sum\limits_{j=1}^k\Big(\frac{1}{(1+\lambda|y-x_j^+|)^{{N-2m}}}+\frac{1}{(1+\lambda|y-x_j^-|)^{{N-2m}}}\Big)\\
  &\times \sum\limits_{j=1}^k\Big(\frac{1}{(1+\lambda|y-x_j^+|)^{\frac{N+2m}{2}+\frac{N+2m}{N-2m}\tau}}+\frac{1}{(1+\lambda|y-x_j^-|)^{\frac{N+2m}{2}+\frac{N+2m}{N-2m}\tau}}\Big)dy\\
  \leq &Ck\|\phi\|_*^{m^*-1}=o\big(\frac{k}{\lambda^{1-\varepsilon}}\big),
\end{align*}
and
\begin{align*}
  I_4
  \leq &C\|\phi\|_*\lambda^N\int_{ \mathbb{R}^N}\sum\limits_{j=1}^k\Big(\frac{1}{(1+\lambda|y-x_j^+|)^{{N+2m}}}+\frac{1}{(1+\lambda|y-x_j^-|)^{{N+2m}}}\Big)\\
  &\times \sum\limits_{j=1}^k\Big(\frac{1}{(1+\lambda|y-x_j^+|)^{\frac{N-2m}{2}+\tau}}+\frac{1}{(1+\lambda|y-x_j^-|)^{\frac{N-2m}{2}+\tau}}\Big)dy\\
  \leq &Ck\|\phi\|_*=o\big(\frac{k}{\lambda^{1-\varepsilon}}\big).
\end{align*}
So we get the result.
\end{proof}

Now we will prove Theorem \ref{th1}.

\vspace{.3cm}

\noindent{\bf Proof of Theorem \ref{th1}.} Through the above discussion, applying \eqref{trans3} and Lemmas \ref{converse}, \ref{conversese} to \eqref{trans2} and \eqref{trans1'''}, we can see that \eqref{con1} and \eqref{con2} are equivalent to
\begin{equation*}
  2k\Big(\frac{1}{\lambda^{2m}}\frac{1}{2\bar{r}^{2m-1}} \frac{\partial\big(\bar{r}^{2m} V(\bar{r},\bar{y}'')\big)}{\partial \bar{r}} \int_{\mathbb{R}^N}U_{0,1}^2dy-\frac{1}{m^*}\bar{r} \frac{\partial Q(\bar{r},\bar{y}'')}{\partial \bar{r}} \int_{\mathbb{R}^N}U_{0,1}^{m^*}dy+o\big(\frac{1}{\lambda^{1-\varepsilon}}\big)\Big)=o\Big(\frac{k}{\lambda^{2m}}\Big),
\end{equation*}
and
\begin{equation*}
  2k\Big(\frac{1}{\lambda^{2m}}\frac{\partial V(\bar{r},\bar{y}'')}{\partial \bar{y}_i}\int_{\mathbb{R}^N}U_{0,1}^2dy-\frac{2}{m^*}\frac{\partial Q(\bar{r},\bar{y}'')}{\partial \bar{y}_i} \int_{\mathbb{R}^N}U_{0,1}^{m^*}dy+o\big(\frac{1}{\lambda^{1-\varepsilon}}\big)\Big)=o\Big(\frac{k}{\lambda^{2m}}\Big),\quad i=4,5,\cdots,N.
\end{equation*}
Therefore, the equations to determine $(\bar{r},\bar{y}'')$ are
\begin{equation}\label{de1}
  \frac{\partial Q(\bar{r},\bar{y}'')}{\partial \bar{r}} =o\big(\frac{1}{\lambda^{1-\varepsilon}}\big),
\end{equation}
and
\begin{equation}\label{de2}
  \frac{\partial Q(\bar{r},\bar{y}'')}{\partial \bar{y}_i}=o\big(\frac{1}{\lambda^{1-\varepsilon}}\big),\quad i=4,5,\cdots,N.
\end{equation}
Moreover, by Lemma \ref{ener1}, the equation to determine $\lambda$ is
\begin{equation}\label{dela}
  -\frac{\tilde{B}_1}{\lambda^{2m+1}}+\frac{\tilde{B}_3k^{N-2m}}{\lambda^{N-2m+1}(\sqrt{1-\bar{h}^2})^{N-2m}}+
  \frac{\tilde{B}_4 k}{\lambda^{N-2m+1}\bar{h}^{N-2m-1}\sqrt{1-\bar{h}^2}}=O\Big(\frac{1}{\lambda^{2m+1+\varepsilon}}\Big),
\end{equation}
where $\tilde{B}_1$, $\tilde{B}_3$, $\tilde{B}_4$ are positive constants.

Let $\lambda=tk^{\frac{N-2m}{N-4m-\nu}}$ with $\nu=N-4M-\iota$, $\iota$ is a small constant, then $t\in [L_0,L_1]$.
From
\eqref{dela}, we have
\begin{equation*}
  -\frac{\tilde{B}_1}{t^{2m+1}}+\frac{\tilde{B}_3M_1^{N-2m}}{t^{N-2m+1-\nu}}=o(1),\quad t\in [L_0,L_1].
\end{equation*}

Define
\begin{equation*}
  F(t,\bar{r},\bar{y}'')=\Big(\nabla _{\bar{r},\bar{y}''}Q(\bar{r},\bar{y}''),-\frac{\tilde{B}_1}{t^{2m+1}}+\frac{\tilde{B}_3M_1^{N-2m}}{t^{N-2m+1-\nu}}\Big).
\end{equation*}
Then, it holds
\begin{equation*}
  deg\Big(F(t,\bar{r},\bar{y}''),[L_0,L_1]\times B_{\lambda^\frac{1}{1-\vartheta}}\big((r_0,y_0'')\big)\Big)=-deg\Big(\nabla _{\bar{r},\bar{y}''}Q(\bar{r},\bar{y}''),B_{\lambda^\frac{1}{1-\vartheta}}\big((r_0,y_0'')\big)\Big)\neq0.
\end{equation*}
Hence, \eqref{de1},\eqref{de2} and \eqref{dela} has a solution $t_k\in [L_0,L_1]$, $(\bar{r}_k,\bar{y}_k'')\in B_{\lambda^\frac{1}{1-\vartheta}}\big((r_0,y_0'')\big)$.
\qed

\section{Proof of Theorem \ref{th2}}\label{four}
In this section, we give a brief proof of Theorem \ref{th2}. We define $\tau=\frac{N-4m}{N-2m}$.  %The main difference between the two proofs is how to deal with problems arising from the distance between points $\{x_j^+\}_{j=1}^k$ and $\{x_j^-\}_{j=1}^k$.

%Therefore, in order to avoid some tedious steps, we just point out some important estimates.

\vspace{.3cm}

\noindent{\bf Proof of Theorem \ref{th2}.} We can verify that
\begin{equation}\label{same}
  \frac{k}{\lambda^\tau}=O(1),\quad \frac{k}{\lambda}=O\Big(\big(\frac{1}{\lambda}\big)^{\frac{2m}{N-2m}}\Big).
\end{equation}
Using \eqref{same} and Lemmas \ref{AppA5}, \ref{AppA6}, we get the same conclusions for problems arising from the distance between points $\{x_j^+\}_{j=1}^k$ and $\{x_j^-\}_{j=1}^k$.

Moreover, by Lemma \ref{AppA4}, we have
\begin{equation}\label{same1}
 |Z_{j,2}^{\pm}|\leq C\lambda^{-\beta_2}Z_{x_j^\pm,\lambda},\quad |Z_{j,l}^{\pm}|\leq C\lambda Z_{x_j^\pm,\lambda},\quad l=3,4,\cdots,N,
\end{equation}
where $\beta_2=\frac{N-4m}{N-2m}$.

Using \eqref{same} and \eqref{same1}, with a similar step in the proof of Theorem \ref{th1} in Sections \ref{two} and \ref{three}, we know
that the
proof of Theorem \ref{th2} has the same reduction structure as that of Theorem \ref{th1} and $u_k$ is a solution of problem
\eqref{pro} if the following equalities hold:
\begin{equation}\label{de1'}
  \frac{\partial Q(\bar{r},\bar{y}'')}{\partial \bar{r}} =o\big(\frac{1}{\lambda^{1-\varepsilon}}\big),
\end{equation}
\begin{equation}\label{de2'}
  \frac{\partial Q(\bar{r},\bar{y}'')}{\partial \bar{y}_i}=o\big(\frac{1}{\lambda^{1-\varepsilon}}\big),\quad i=4,5,\cdots,N,
\end{equation}
\begin{equation}\label{dela'}
  -\frac{\tilde{B}_1}{\lambda^{2m+1}}+\frac{\tilde{B}_3k^{N-2m}}{\lambda^{N-2m+1}(\sqrt{1-\bar{h}^2})^{N-2m}}+
  \frac{\tilde{B}_4 k}{\lambda^{N-2m+1}\bar{h}^{N-2m-1}\sqrt{1-\bar{h}^2}}=O\Big(\frac{1}{\lambda^{2m+1+\varepsilon}}\Big).
\end{equation}

Let $\lambda=tk^{\frac{N-2m}{N-4m}}$, then $t\in [L_0',L_1']$.
Next, we discuss the main items in \eqref{dela'}.

{\bf Case 1.} If $\bar{h}\rightarrow A\in (0,1)$, then $(\lambda^{\frac{N-4m}{N-2m}}\bar{h})^{-1}\rightarrow0$ as $\lambda\rightarrow\infty$,
from
\eqref{dela'}, we have
\begin{equation*}
  -\frac{\tilde{B}_1}{t^{2m+1}}+\frac{\tilde{B}_3}{t^{N-2m+1}(\sqrt{1-A^2})^{N-2m}}=o(1),\quad t\in [L_0',L_1'].
\end{equation*}
Define
\begin{equation*}
  F(t,\bar{r},\bar{y}'')=\Big(\nabla _{\bar{r},\bar{y}''}Q(\bar{r},\bar{y}''),-\frac{\tilde{B}_1}{t^{2m+1}}+\frac{\tilde{B}_3}{t^{N-2m+1}(\sqrt{1-A^2})^{N-2m}}\Big).
\end{equation*}
Then, it holds
\begin{equation*}
  deg\Big(F(t,\bar{r},\bar{y}''),[L_0,L_1]\times B_{\lambda^\frac{1}{1-\vartheta}}\big((r_0,y_0'')\big)\Big)=-deg\Big(\nabla _{\bar{r},\bar{y}''}Q(\bar{r},\bar{y}''),B_{\lambda^\frac{1}{1-\vartheta}}\big((r_0,y_0'')\big)\Big)\neq0.
\end{equation*}
Hence, \eqref{de1'}, \eqref{de2'} and \eqref{dela'} has a solution $t_k\in [L_0',L_1']$, $(\bar{r}_k,\bar{y}_k'')\in B_{\lambda^\frac{1}{1-\vartheta}}\big((r_0,y_0'')\big)$.

{\bf Case 2.} If $\bar{h}\rightarrow 0$ and $(\lambda^{\frac{N-4m}{N-2m}}\bar{h})^{-1}\rightarrow0$ as $\lambda\rightarrow\infty$,
from
\eqref{dela'}, we have
\begin{equation*}
  -\frac{\tilde{B}_1}{t^{2m+1}}+\frac{\tilde{B}_3}{t^{N-2m+1}}=o(1),\quad t\in [L_0',L_1'].
\end{equation*}
Define
\begin{equation*}
  F(t,\bar{r},\bar{y}'')=\Big(\nabla _{\bar{r},\bar{y}''}Q(\bar{r},\bar{y}''),-\frac{\tilde{B}_1}{t^{2m+1}}+\frac{\tilde{B}_3}{t^{N-2m+1}}\Big).
\end{equation*}
Then, we can find a solution $(t_k,\bar{r}_k,\bar{y}_k'')$ of \eqref{de1'}, \eqref{de2'} and \eqref{dela'} as before.
%\begin{equation*}
%  deg\Big(F(t,\bar{r},\bar{y}''),[L_0',L_1']\times B_\vartheta\big((r_0,y_0'')\big)\Big)=-deg\Big(\nabla _{\bar{r},\bar{y}''}\big(\bar{r}^{2m}V(\bar{r},\bar{y}'')\big),B_\vartheta\big((r_0,y_0'')\big)\Big)\neq0.
%\end{equation*}
%Hence, \eqref{de1'},\eqref{de2'} and \eqref{dela'} has a solution $t_k\in [L_0',L_1']$, $(\bar{r}_k,\bar{y}_k'')\in B_\vartheta\big((r_0,y_0'')\big)$.

{\bf Case 3.} If $\bar{h}\rightarrow 0$ and $(\lambda^{\frac{N-4m}{N-2m}}\bar{h})^{-1}\rightarrow A\in (C_1,M_2)$ for some positive constant $C_1$ as $\lambda\rightarrow\infty$,
from
\eqref{dela'}, we have
\begin{equation*}
  -\frac{\tilde{B}_1}{t^{2m+1}}+\frac{\tilde{B}_3}{t^{N-2m+1}}+\frac{\tilde{B}_4A^{N-2m-1}}{t^{2m+1+\frac{N-4m}{N-2m}}}=o(1),\quad t\in [L_0',L_1'].
\end{equation*}
%Define
%\begin{equation*}
%  F(t,\bar{r},\bar{y}'')=\Big(\nabla _{\bar{r},\bar{y}''}\big(\bar{r}^{2m}V(\bar{r},\bar{y}'')\big),-\frac{B_1}{t^{2m+1}}V(\bar{r},\bar{y}'')+\frac{B_3}{t^{N-2m+1}}\Big).
%\end{equation*}
%Then, we can find a solution $(t_k,\bar{r}_k,\bar{y}_k'')$ of \eqref{de1'},\eqref{de2'} and \eqref{dela'} as before.
Since $N-2m+1$ and $2m+1+\frac{N-4m}{N-2m}$ are strictly greater than $2m+1$, there exists a solution of \eqref{de1'}, \eqref{de2'} and \eqref{dela'} as before.
\qed

\vspace{.5cm}

\begin{appendices}

\section{{\bf Some basic estimates}}\label{AppA}
In this section, we give some basic estimates.
\begin{lemma}\cite[Lemma B.1]{WY1}\label{AppA1}
For $y,x_i,x_j\in \mathbb{R}^N$, $i\neq j$, let
\begin{equation*}
  g_{ij}(y)=\frac{1}{(1+|y-x_i|)^{\kappa_1}}\frac{1}{(1+|y-x_j|)^{\kappa_2}},
\end{equation*}
where $\kappa_1,\kappa_2\geq 1$ are constants. Then for any constant $0<\sigma\leq \min\{\kappa_1,\kappa_2\}$, there exists a constant $C>0$ such that
\begin{equation*}
  g_{ij}(y)\leq \frac{C}{|x_i-x_j|^\sigma}\Big(\frac{1}{(1+|y-x_i|)^{\kappa_1+\kappa_2-\sigma}}+\frac{1}{(1+|y-x_j|)^{\kappa_1+\kappa_2-\sigma}}\Big).
\end{equation*}
\end{lemma}

%\begin{lemma}\cite[Lemma 2.2]{GL}\label{AppA2}
%For any constant $0<\sigma<N-2m$, there exists a constant $C>0$ such that
%\begin{equation*}
%  \int_{\mathbb{R}^N}\frac{1}{|y-z|^{N-2m}}\frac{1}{(1+|z|)^{2m+\sigma}}dz\leq \frac{C}{(1+|y|)^\sigma}.
%\end{equation*}
%\end{lemma}

%\begin{lemma}\label{AppA3}
%Assume that $N> 4m+1$, $0<\tau<2$, then there exists a small constant $\sigma>0$ such that
%\begin{equation*}
%  \int_{\mathbb{R}^N}\frac{1}{|y-z|^{N-2m}}Z_{\bar{r},\bar{h},\bar{y}'',\lambda}^{m^*-2}(z)\sum\limits_{j=1}^k\frac{1}{(1+\lambda|z-x_j^+|)^{\frac{N-2m}{2}+\tau}}dz\leq C\sum\limits_{j=1}^k\frac{1}{(1+\lambda|y-x_j^+|)^{\frac{N-2m}{2}+\tau+\sigma}},
%\end{equation*}
%and
%\begin{equation*}
%  \int_{\mathbb{R}^N}\frac{1}{|y-z|^{N-2m}}Z_{\bar{r},\bar{h},\bar{y}'',\lambda}^{m^*-2}(z)\sum\limits_{j=1}^k\frac{1}{(1+\lambda|z-x_j^-|)^{\frac{N-2m}{2}+\tau}}dz\leq C\sum\limits_{j=1}^k\frac{1}{(1+\lambda|y-x_j^-|)^{\frac{N-2m}{2}+\tau+\sigma}}.
%\end{equation*}
%\end{lemma}
%\begin{proof}
%The proof is similar to \cite[Lemma 2.3]{GL}, so we omit it here.
%\end{proof}

\begin{lemma}\label{AppA4}
As $\lambda\rightarrow\infty$, we have
\begin{equation*}
  \frac{\partial U_{x_j^\pm,\lambda}}{\partial \lambda}=O(\lambda^{-1}U_{x_j^\pm,\lambda})+O(\lambda U_{x_j^\pm,\lambda})\frac{\partial \sqrt{1-\bar{h}^2}}{\partial \lambda}+
  O(\lambda U_{x_j^\pm,\lambda})\frac{\partial \bar{h}}{\partial \lambda}.
\end{equation*}
Hence, if $\sqrt{1-\bar{h}^2}=C\lambda^{-\beta_1}$ with $0<\beta_1<1$, we have
\begin{equation*}
  \Big|\frac{\partial U_{x_j^\pm,\lambda}}{\partial \lambda}\Big| \leq C\frac{U_{x_j^\pm,\lambda}}{\lambda^{\beta_1}}.
\end{equation*}
If $\bar{h}=C\lambda^{-\beta_2}$ with $0<\beta_2<1$, then we have
\begin{equation*}
  \Big|\frac{\partial U_{x_j^\pm,\lambda}}{\partial \lambda}\Big| \leq C\frac{U_{x_j^\pm,\lambda}}{\lambda^{\beta_2}}.
\end{equation*}
\end{lemma}
\begin{proof}
The proof is standard, we omit it.
\end{proof}

\begin{lemma}\cite[Lemma A.2]{DHWW}\label{AppA5}
For any $\gamma>0$, there exists a constant $C>0$ such that
\begin{equation*}
  \sum\limits_{j=2}^k\frac{1}{|x_j^+-x_1^+|^\gamma}\leq \frac{Ck^\gamma}{(\bar{r}\sqrt{1-\bar{h}^2})^\gamma}\sum\limits_{j=2}^k\frac{1}{(j-1)^\gamma}\leq
  \left\{
  \begin{array}{ll}
  \frac{Ck^\gamma}{(\bar{r}\sqrt{1-\bar{h}^2})^\gamma},\quad \gamma>1;\\
  \frac{Ck^\gamma \log k}{(\bar{r}\sqrt{1-\bar{h}^2})^\gamma},\quad \gamma=1;\\
  \frac{Ck}{(\bar{r}\sqrt{1-\bar{h}^2})^\gamma},\quad \gamma<1,
    \end{array}
    \right.
\end{equation*}
and
\begin{equation*}
   \sum\limits_{j=1}^k\frac{1}{|x_j^--x_1^+|^\gamma}\leq  \sum\limits_{j=2}^k\frac{1}{|x_j^+-x_1^+|^\gamma}+\frac{C}{(\bar{r}\bar{h})^\gamma}.
\end{equation*}
\end{lemma}

\begin{lemma}\cite[Lemma A.6]{CW}\label{AppA6}
Assume that $N>4m+1$, as $k\rightarrow\infty$, we have
\begin{equation*}%\label{A.1}\tag{A.1}
  \sum\limits_{j=2}^k\frac{1}{|x_j^+-x_1^+|^{N-2m}}=\frac{A_1k^{N-2m}}{(\sqrt{1-\bar{h}^2})^{N-2m}}\Big(1+o\big(\frac{1}{k}\big)\Big),
\end{equation*}
and if $\frac{1}{\bar{h}k}=o(1)$, then
\begin{equation*}%\label{A.2}\tag{A.2}
   \sum\limits_{j=1}^k\frac{1}{|x_j^--x_1^+|^{N-2m}}=\frac{A_2k}{\bar{h}^{N-2m-1}(\sqrt{1-\bar{h}^2})}\Big(1+o\big(\frac{1}{\bar{h}k}\big)\Big)+O\Big(\frac{1}{(\sqrt{1-\bar{h}^2})^{N-2m}}\Big),
\end{equation*}
or else, $\frac{1}{\bar{h}k}=C$, then
\begin{equation*}%\label{A.3}\tag{A.3}
   \sum\limits_{j=1}^k\frac{1}{|x_j^--x_1^+|^{N-2m}}=\Big(\frac{A_3k}{\bar{h}^{N-2m-1}},\frac{A_4k}{\bar{h}^{N-2m-1}}\Big),
\end{equation*}
where $A_1$, $A_2$, $A_3$ and $A_4$ are some positive constants.
\end{lemma}

Suppose that $u$ and $v$ are two smooth functions in a given bounded domain $\Omega$,
define the following two bilinear functionals
\begin{equation*}
  L_{1}(u,v)=\int_\Omega \Big((-\Delta)^m u \langle y,\nabla v\rangle+(-\Delta)^m v \langle y,\nabla u\rangle\Big) dy,
\end{equation*}
\begin{equation*}
  L_{2,i}(u,v)=\int_\Omega \Big((-\Delta)^m u \frac{\partial v}{\partial y_i}+(-\Delta)^m v \frac{\partial u}{\partial y_i}\Big)dy,\quad i=1,2,\cdots,N.
\end{equation*}
We have
\begin{lemma}\cite[Proposition 3.3]{GPY}\label{AppA8}
For any integer $m>0$, there exists a function $f_{m}(u,v)$ such that
\begin{equation*}
  L_{1}(u,v)=\int_{\partial \Omega}f_{m}(u,v)dy-\frac{N-2m}{2}\int_\Omega \Big(v (-\Delta)^m u+u (-\Delta)^m v\Big)dy.
\end{equation*}
Moreover, $f_{m}(u,v)$ has the form
\begin{equation*}
  f_{m}(u,v)=\sum\limits_{j=1}^{2m-1}\bar{l}_{j}(y,\nabla ^j u,\nabla^{2m-j}v)+\sum\limits_{j=0}^{2m-1}\tilde{l}_{j}(\nabla ^j u,\nabla^{2m-j-1}v),
\end{equation*}
where $\bar{l}_{j}(y,\nabla ^j u,\nabla^{2m-j}v)$ and $\tilde{l}_{j}(\nabla ^j u,\nabla^{2m-j-1}v)$ are linear in each component.
\end{lemma}
\begin{lemma}\cite[Proposition 3.1]{GPY}\label{AppA9}
For any integer $m>0$, there exists a function $g_{m,i}(u,v)$ such that
\begin{equation*}
  L_{2,i}(u,v)=\int_{\partial \Omega}g_{m,i}(u,v)dy.
\end{equation*}
Moreover, $g_{m,i}(u,v)$ has the form
\begin{equation*}
  g_{m,i}(u,v)=\sum\limits_{j=1}^{2m-1}l_{j,i}(\nabla ^j u,\nabla^{2m-j}v),
\end{equation*}
where $l_{j,i}(\nabla ^j u,\nabla^{2m-j}v)$ is bilinear in $\nabla ^j u$ and $\nabla ^{2m-j} u$.
\end{lemma}

\medskip
\subsection*{Acknowledgments}
The research has been supported by Natural Science Foundation of Chongqing, China  CSTB2024 NSCQ-LZX0038.

\subsection*{Statements and Declarations}
{\bf Conflict of interest} On behalf of all authors, the corresponding author states that there
is no conflict of interest. The manuscript has no associated data.

\end{appendices}

\end{document}